\documentclass[12pt,reno]{amsart}
\usepackage{amssymb}
\usepackage{amsmath}
\usepackage{enumitem}
\usepackage{mathrsfs}
\usepackage[english]{babel}
\usepackage[all]{xy}
\usepackage{hyperref}
\usepackage{marginnote}
\usepackage{array, graphicx}
\usepackage{pgf,tikz,amsfonts}

\usetikzlibrary{arrows}

\usepackage{stmaryrd}

\usepackage[margin=1.25in]{geometry}

\RequirePackage{mathrsfs} 

\newtheorem{theorem}{Theorem}[section]
\newtheorem{lemma}[theorem]{Lemma}
\newtheorem{proposition}[theorem]{Proposition}
\newtheorem{corollary}[theorem]{Corollary}
\newtheorem{conjecture}[theorem]{Conjecture}

\theoremstyle{definition}
\newtheorem*{ack}{Acknowledgements}
\newtheorem*{con}{Conventions}
\newtheorem{remark}[theorem]{Remark}
\newtheorem{example}[theorem]{Example}
\newtheorem{definition}[theorem]{Definition}
\newtheorem{question}{Question}

\numberwithin{equation}{section} \numberwithin{figure}{section}

\DeclareMathOperator{\Aut}{Aut}

\DeclareMathOperator{\Spec}{Spec}

\DeclareMathOperator{\an}{an}

\DeclareMathOperator{\Hom}{Hom}

\newcommand{\Qbar}{\overline{\QQ}}

\newcommand\ZZ{\mathbb{Z}}

\newcommand\QQ{\mathbb{Q}}

\newcommand\CC{\mathbb{C}}

\newcommand\OO{\mathcal{O}}

\usepackage{color}

\title[The Lang-Vojta conjectures on projective pseudo-hyperbolic varieties]{The Lang-Vojta    conjectures on projective pseudo-hyperbolic varieties}

\author{Ariyan Javanpeykar}
\address{Ariyan Javanpeykar \\
Institut f\"{u}r Mathematik\\
Johannes Gutenberg-Universit\"{a}t Mainz\\
Staudingerweg 9, 55099 Mainz\\
Germany.}
\email{peykar@uni-mainz.de}

\subjclass[2010]
{14G99 
(11G35,  
14G05,  
32Q45)} 

\keywords{Integral points, hyperbolicity, automorphisms,  moduli spaces, dynamical systems, Green-Griffiths-Lang-Vojta conjecture.}

\begin{document}

\maketitle
 \tableofcontents

\thispagestyle{empty}

\section{Introduction}

These notes grew out of a mini-course given from May 13th to May 17th at UQ\`AM in Montr\'eal during a workshop on Diophantine Approximation and  Value Distribution Theory.

\subsection{What is in these notes?}  We   start with an overview of Lang-Vojta's conjectures on pseudo-hyperbolic  \emph{projective} varieties. These conjectures relate various  different notions of hyperbolicity.    We  start with Brody hyperbolicity and   discuss conjecturally related notions of hyperbolicity in arithmetic geometry and algebraic geometry in subsequent sections. We slowly work our way towards the most general version of Lang-Vojta's conjectures and provide a summary of all the conjectures in Section \ref{section:conjectures}.     

  After having explained the main conjectures with the case of curves and closed subvarieties of abelian varieties as our guiding principle, we collect recent advances on Lang-Vojta's conjectures and present these in a unified manner. These results are concerned with endomorphisms of hyperbolic varieties, moduli spaces of maps into a hyperbolic variety, and also the behavior of hyperbolicity in families of varieties.  The   results presented in these sections are proven in \cite{vBJK, JAut, JKa, JVez, JXie}. 
  
We also  present results on the Shafarevich conjecture for smooth hypersurfaces obtained in joint work with Daniel Litt \cite{JLitt}. These are motivated by Lawrence-Venkatesh's recent breakthrough on the non-density of integral points on the moduli space of   hypersurfaces \cite{LV}, and are in accordance with Lang-Vojta's conjecture for \emph{affine} varieties.  Our results in this section are proven using methods from Hodge theory, and are loosely related to Bakker-Tsimerman's chapter in this book \cite{BakkerTsimermanBook}.
  
  In the final section we sketch a proof of the fact that being groupless is a Zariski-countable open condition, and thus in particular stable under generization. To prove this, we follow \cite{JVez} and introduce a non-archimedean notion of hyperbolicity. We then state  a non-archimedean analogue of the Lang-Vojta conjectures which we prove under suitable assumptions. These results suffice to prove   that grouplessness is stable under generization.

\subsection{Anything new in these notes?}  The main contribution of these notes is the systematic presentation and comparison between different notions of hyperbolicity, and their ``pseudofications''. As it is intended to be a broad-audience introduction   to the Lang-Vojta conjectures, it contains all definitions and well-known relations between these.  Also, Lang-Vojta's original conjectures are often only stated for varieties over $\Qbar$, and we  propose natural extensions of their conjectures to varieties over arbitrary algebraically closed fields of characteristic zero.  We also define for each notion appearing in the conjecture the relevant ``exceptional locus'' (which Lang only does for some notions of hyperbolicity in \cite{Lang2}).

The final version of Lang-Vojta's conjecture as stated in Section \ref{section:conjectures} does not appear anywhere in the literature explicitly.
Furthermore, the section on groupless varieties (Section \ref{section:groupless}) contains simple proofs that do not appear explicitly elsewhere.  Also, we  have included a thorough discussion of the a priori difference between  being arithmetically hyperbolic and Mordellic for a projective variety  in Section \ref{section:ps_mordell}.  This difference is not addressed anywhere else in the literature.

 \subsection{Rational points over function fields} We have not included any discussion of rational points on projective varieties over function fields of smooth connected curves over a field $k$, and unfortunately ignore the relation to Lang-Vojta's conjecture throughout these notes. 

\subsection{Other relevant literature} 
Lang stated his conjectures in \cite{Lang2}; see    also \cite[Conjecture~XV.4.3]{CornellSilverman} and \cite[\S0.3]{Abr}. In \cite[Conj.~4.3]{Vojta3} Vojta extended this conjecture to quasi-projective varieties. In   \cite{Lang2} Lang  ``pseudofied'' the notion of Brody hyperbolicity.  Here he was inspired by   Kiernan-Kobayashi's extension of the notion of Kobayashi hyperbolicity introduced in \cite{KiernanKobayashi}. 

There are several beautiful surveys of the Green-Griffiths and Lang-Vojta conjectures. We mention \cite{Corvaja1, Corvaja2, Turchet, DiverioR, Gasbarri,  VojtaLangExc}.

The  first striking consequence of  Lang-Vojta's conjecture was obtained by Caporaso-Harris-Mazur \cite{CHM}. Their results were further investigated by Abramovich, Ascher-Turchet,   Hassett, and Voloch; see \cite{Abr, AbrCor,  AbrMat, AbrVol, AscherTurchet, HassettCor}.

Campana's conjectures   provide a complement to Lang-Vojta's conjectures, and first appeared in \cite{Campana, CampanaOr}; see also   Campana's chapter in this book \cite{CampanaBook}. In a nutshell, the ``opposite'' of being pseudo-hyperbolic (in any sense of the word ``hyperbolic'') is conjecturally captured by Campana's notion of a ``special'' variety.
  
  \begin{ack}
We thank Marc-Hubert Nicole for his patience and   guidance when writing these notes. 

We  would also like to thank Carlo Gasbarri,  Nathan Grieve, Aaron Levin, Steven Lu, Marc-Hubert Nicole, Erwan Rousseau, and Min Ru for organizing the research workshop \emph{Diophantine Approximation and Value Distribution Theory} and giving us the opportunity to give the mini-course on which these notes are based.  

 These notes would not exist without the input of  Kenneth Ascher,  Raymond van Bommel, Ljudmila Kamenova, Robert Kucharczyk, Daniel Litt, Daniel Loughran, Siddharth Mathur, Jackson Morrow,  Alberto Vezzani, and Junyi Xie. 
 
 We are grateful to the organizers Philipp Habegger, Ronan Terpereau, and Susanna Zimmermann of the 7th Swiss-French workshop in Algebraic Geometry in 2018 in Charmey for giving me the opportunity to speak on hyperbolicity of moduli spaces. Part of these notes are also based on the mini-course I gave in Charmey.
 
 We are especially grateful to Raymond van Bommel for providing the diagrams and figure in Section \ref{section:conjectures}.
 
We   gratefully acknowledge support from   SFB/Transregio 45.
\end{ack}

   \begin{con} Throughout these notes, we will let $k$  be an algebraically closed field of characteristic zero. 
If $X$ is a locally finite type scheme over $\mathbb{C}$, we let $X^{\an}$ be the associated complex-analytic space  \cite[Expose~XII]{SGA1}. If $K$ is a field, then  a variety over  $K$ is a finite type separated     $K$-scheme.  

 If $X$ is a variety over a field $K$ and $L/K$ is a field extension, then $X_L:= X\times_{\Spec K} \Spec L$ will denote the base-change  of $X\to \Spec K$ along  $\Spec L\to \Spec K$. More generally, if $R\to R'$ is an extension of rings and $X$ is  a scheme over $R$, we let $X_{R'}$ denote $X\times_{\Spec R} \Spec R'$.

 If $K$ is a number field and $S$ is a finite set of finite places of $K$, then  $\OO_{K,S}$ will denote the ring of $S$-integers of $K$.
 \end{con}

\section{Brody hyperbolicity}\label{section:brody}
We start  with the classical notion of Brody hyperbolicity for complex varieties.   
 
 \begin{definition}
 A complex-analytic space $X$ is \emph{Brody hyperbolic} if every holomorphic map $\mathbb{C}\to X$ is constant. A locally finite type  scheme $X$ over $\CC$ is \emph{Brody hyperbolic} if $X^{\an}$ is Brody hyperbolic.
 \end{definition}
 
 If $X$ is a complex-analytic space, then a non-constant holomorphic map $\CC\to X$ is commonly referred to   as an entire curve in $X$. Thus, to say that $X$ is Brody hyperbolic is to say that $X$ has no entire curves.
 
 We recall that a   complex-analytic space $X$ is \emph{Kobayashi hyperbolic} if Kobayashi's pseudometric on $X$ is a metric \cite{KobayashiBook}. It is a fundamental result of Brody that a \emph{compact} complex-analytic space $X$ is Brody hyperbolic if and only if it is Kobayashi hyperbolic; see \cite[Theorem~3.6.3]{KobayashiBook}.   
 
 \begin{remark}[Descending Brody hyperbolicity]\label{remark:cw_brody}
 Let $X\to Y$  be a proper \'etale (hence finite) morphism of varieties over $\CC$. It is not hard to show that $X$ is Brody hyperbolic if and only if $Y$ is Brody hyperbolic. (It is crucial that $X\to Y$ is finite \textbf{and} \'etale.)
 \end{remark}
 
 Fundamental results in complex analysis lead to the following classification of Brody hyperbolic projective curves.

 \begin{theorem}[Liouville, Riemann, Schwarz, Picard]\label{thm:riemann}
 Let $X$ be a smooth  projective connected curve over $\CC$. Then $X$ is  Brody hyperbolic if and only if $\mathrm{genus}(X)\geq 2$.  
 \end{theorem}

 More generally, a smooth quasi-projective connected curve $X$ over $\CC$ is Brody hyperbolic if and only if $X$ is not isomorphic to $\mathbb{P}^1_{\CC}$, $\mathbb{A}^1_{\CC}$, $\mathbb{A}^1_{\CC}\setminus \{0\}$, nor a smooth proper connected genus one curve over $\CC$.
 
 \begin{remark}\label{remark:abvar_Brody} It is implicit in Theorem \ref{thm:riemann} that elliptic curves are not Brody hyperbolic. More generally, a non-trivial abelian variety $A$ of dimension $g$ over $\CC$ is not Brody hyperbolic, as its associated complex-analytic space is uniformized by $\mathbb{C}^g$.  Since $A$ even has a dense entire curve, one can consider $A$ to be as far as possible from being Brody hyperbolic.   We mention that Campana conjectured that a projective variety has a dense entire curve if and only if it is ``special''.   We  refer the reader to Campana's article in this book for a further discussion of Campana's conjecture  \cite{CampanaBook}.
 \end{remark}
 
By Remark \ref{remark:abvar_Brody},  an obvious obstruction to a projective variety $X$ over $\mathbb{C}$ being Brody hyperbolic is that it contains an abelian variety.  The theorem of Bloch--Ochiai--Kawamata says that this is the only obstruction if $X$   can be embedded into an abelian variety (see \cite{Kawamata}).

 \begin{theorem}[Bloch--Ochiai--Kawamata]\label{thm:bok}
Let $X$ be a closed subvariety of an abelian variety $A$ over $\CC$. Then $X$ is Brody hyperbolic if and only if $X$ does not contain the translate of a positive-dimensional abelian subvariety of $A$.
 \end{theorem}
 
Throughout these notes, we mostly focus on closed subvarieties of abelian varieties, as in this case the results concerning Lang-Vojta's conjectures are   complete; see  Section \ref{section:questions} for details.
 
 The theorem of Bloch--Ochiai--Kawamata has been pushed further by work of Noguchi--Winkelmann-Yamanoi; see \cite{NWY1, NWY2, NWY3, Yamanoi1, Yamanoi2}.   Other examples of Brody hyperbolic varieties can be constructed as quotients of bounded domains, as we explain now.

\begin{remark}[Bounded domains]\label{remark:bd}
 Let $D$ be a bounded domain in the affine space $\mathbb{C}^N$, and let $X$ be a   reduced  connected locally finite type  scheme over $\CC$. Then, any holomorphic map $X^{\an}\to D$ is constant; see \cite[Remark~2.9]{JVez} for a detailed proof. In particular, the complex-analytic space $D$ is Brody hyperbolic (take $X=\mathbb{A}^1_{\CC}$).
 \end{remark}
 
It follows from Remark \ref{remark:bd} that a (good) quotient of a bounded domain is Brody hyperbolic. This observation applies to locally symmetric varieties, Shimura varieties, and thus moduli spaces of abelian varieties. We conclude this section by recording the fact that the moduli space of abelian varieties (defined appropriately) is a Brody hyperbolic variety.
 \begin{example}\label{exa:bd_Ag} Let $g\geq 1$ and let $N\geq 3$ be integers. Then, 
 the (fine) moduli space of $g$-dimensional  principally polarized abelian varieties with level $N$ structure is a smooth quasi-projective variety over $\CC$ which is Brody hyperbolic. Indeed, its universal cover is biholomorphic to a bounded domain in $\mathbb{C}^{g(g+1)/2}$, so that we can apply Remark \ref{remark:bd}. (As the coarse moduli space of elliptic curves is given by the $j$-line $\mathbb{A}^1_{\mathbb{C}}$, we see that it is not Brody hyperbolic.  This is the reason for which we consider  the moduli space of abelian varieties with level structure.)
 \end{example}

 \section{Mordellic varieties}\label{section:Mordell}
 What should correspond to being Brody hyperbolic in arithmetic geometry? Lang was the first to  propose that a ``Mordellic'' projective variety over $\Qbar$ should be Brody hyperbolic (over the complex numbers). Roughly speaking, a projective variety over $\Qbar$ is Mordellic if it has only finitely many rational points in any fixed number field. To make this more precise, one has to choose models (see Definition \ref{defn:mor} below).   Conversely, a projective variety over a number field which is Brody hyperbolic  (over the complex numbers) should be Mordellic. In this section we will present this conjecture of Lang.
 
Throughout this section, we  let $k$ be an algebraically closed field of characteristic zero.   We first clarify what is meant with a model.

\begin{definition}\label{defn:mor}
Let $X$ be a finite type separated scheme over $k$ and let $A\subset k$ be a subring. A \emph{ model for $X$ over $A$} is a pair $(\mathcal{X},\phi)$ with $\mathcal{X}\to \Spec A$ a finite type separated scheme and $\phi:\mathcal{X}_k  \xrightarrow{\sim}  X$ an isomorphism of schemes over $k$. We will often omit $\phi$  from our notation.
\end{definition}

\begin{remark}
What constitutes the data of a model for $X$ over $A$? To explain this, let $X$ be an affine variety over $\CC$, say $X = \Spec R$. Note that the coordinate ring $R$ of $X$ is a finite type $\CC$-algebra.  Suppose  that $X$ is given by  the zero locus of polynomials $f_1,\ldots, f_r$ with coefficients in a subring $A$, so that  $R \cong \CC[x_1,\ldots,x_n]/(f_1,\ldots,f_r)$. Then $\mathcal{R} := A[x_1,\ldots,x_n]/(f_1,\ldots,f_r)\subset R$ is a finitely generated $A$-algebra    and  $\mathcal{R}\otimes_A \CC = R$. That is, if $\mathcal{X} = \Spec \mathcal{R}$, then $\mathcal{X}$   is a model for $X$ over $A$. We will be interested in studying $A$-valued points on $\mathcal{X}$. We follow common notation and let  $\mathcal{X}(A) $ denote the set $\mathrm{Hom}_A(\Spec A,\mathcal{X})$. Note that  $\mathcal{X}(A)$ is the set of solutions in $A$ of the polynomial system of equations $f_1 = \ldots = f_r = 0$.
\end{remark}

With the notion of model now clarified, we are ready to define what it means for a proper variety to be Mordellic. We leave the more general definition for non-proper varieties to the end of this section.

\begin{definition} A proper scheme  $X$ over $k$ is \emph{Mordellic over $k$} 
(or: \emph{has-only-finitely-many-rational-points over $k$})
if,  for every finitely generated subfield $K\subset k$ and every (proper) model $\mathcal{X}$ over $K$, the set $\mathcal{X}(K) := \Hom_K(\Spec K, \mathcal{X})$ is finite.
\end{definition} 

\begin{remark}[Independence of models]
We point out that the finiteness property required for a projective variety to be Mordellic can also be tested on a fixed model. 
That is, a proper scheme  $X$ over $k$ is Mordellic over $k$ if and only if 
there is a finitely generated subfield $K\subset k$ and a proper model $\mathcal{X}$ for $X$ over $K$ such that  for all finitely generated subfields $L\subset k$ containing $K$, the set $\mathcal{X}(L):=\Hom_K(\Spec L, \mathcal{X})$ is finite.
 \end{remark}

We note that Mordellicity (just like Brody hyperbolicity) descends along finite \'etale morphisms (Remark \ref{remark:cw_brody}).

\begin{remark}[Descending Mordellicity]\label{remark:cw_mordell}
Let $X\to Y$ be a finite \'etale morphism of projective varieties over $k$. Then it follows from the Chevalley-Weil theorem  that $X$ is Mordellic over $k$ if and only if $Y$ is Mordellic over $k$; see Theorem \ref{thm:cw} for a proof (of a more general result).
\end{remark}

 It is clear that $\mathbb{P}^1_k$ is not Mordellic, as $\mathbb{P}^1(\QQ)$ is dense.  A deep theorem of Faltings leads to the following classification of projective Mordellic curves. If $k=\Qbar$, then this theorem is proven in \cite{Faltings2}. The statement below is proven in \cite{FaltingsComplements} (see also \cite{Szpiroa}).

\begin{theorem}[Faltings]\label{thm:faltings_curves}
Let $X$ be a smooth projective connected curve over $k$. Then $X$ is Mordellic over $k$ if and only if  $\mathrm{genus}(X) \geq  2$.
\end{theorem}

Recall that abelian varieties are very far from being Brody hyperbolic (Remark \ref{remark:abvar_Brody}). The following remark says that abelian varieties are also very far from being Mordellic.

 \begin{remark}[Hassett-Tschinkel] \label{remark:ht} It is not at all    obvious that a smooth projective connected curve of genus one over $\Qbar$ is not Mordellic. Indeed, it is not an obvious fact that an elliptic curve over a number field $K$ has positive rank over some finite field extension of $K$, although this is certainly true and can be proven in many different ways.  In fact, by a theorem of Hassett-Tschinkel (see \cite[\S 3.1]{JAut} or \cite[\S3]{HassettTschinkel}), if $A$ is an abelian variety over $k$, then there is a  finitely generated subfield $K\subset k$ and an abelian variety $\mathcal{A}$ over $K$ with $\mathcal{A}_k\cong A$ such that $\mathcal{A}(K)$ is dense in $A$. This theorem is not hard to prove when $k$ is uncountable but requires non-trivial arguments otherwise. Thus, if $\dim A\neq 0$, then one can consider the abelian variety $A$ to be as far as possible from being Mordellic. This statement is to be compared with the conclusion of Remark \ref{remark:abvar_Brody}.
 \end{remark}

By Hassett--Tschinkel's theorem (Remark \ref{remark:ht}), an obvious obstruction to a projective variety $X$ over $k$ being Mordellic is that it contains an abelian variety. The following theorem of Faltings says that this is the only obstruction if $X$   can be embedded into an abelian variety; see \cite{FaltingsLang}.

\begin{theorem}[Faltings]\label{thm:faltings}
Let $X$ be a closed subvariety of an abelian variety $A$ over $k$. Then $X$ is Mordellic over $k$  if and only if $X$ does not contain the translate of a positive-dimensional abelian subvariety of $A$.
\end{theorem}

There are strong similarities between the statements in the previous section and the current section. These similarities  (and a healthy dose of optimism)  lead to the first version of the Lang-Vojta conjecture. To state this conjecture, let us say that a variety $X$ over $k$   is \emph{strongly-Brody hyperbolic over $k$} if, for every subfield $k_0\subset k$, every model $\mathcal{X}$ for $X$ over $k_0$, and every embedding $k_0\to \CC$, the variety $\mathcal{X}_{\CC}$ is Brody hyperbolic.

\begin{conjecture}[Weak Lang-Vojta, I]
Let $X$ be a projective variety over $k$. Then $X$ is Mordellic over $k$ if and only if $X$ is strongly-Brody hyperbolic over $k$.
\end{conjecture}

As stated, this conjecture does not predict that, if $X$ is a projective  Brody hyperbolic variety over $\CC$, then every conjugate of $X$ is Brody hyperbolic. We state this conjecture separately.

\begin{conjecture}[Conjugates of Brody hyperbolic varieties]\label{conj:conjugates}
If $X$ is a variety over $k$. Then $X$ is strongly-Brody hyperbolic over $k$ if and only if  there is a subfield $k_0\subset k$, a model $\mathcal{X}$ for $X$ over $k_0$, and an embedding $k_0\to \CC$ such that the variety $\mathcal{X}_{\CC}$ is Brody hyperbolic.
\end{conjecture} 

Concretely,  Conjecture \ref{conj:conjugates} says that, if $X$ is a Brody hyperbolic variety over $\mathbb{C}$ and $\sigma$ is a field automorphism of $\CC$, then the $\sigma$-conjugate $X^{\sigma}$ of $X$ is again Brody hyperbolic.

We briefly discuss the notion of Mordellicity for quasi-projective (not necessarily proper) schemes. We will also comment on this more general notion in Section \ref{section:ps_mordell}. This notion appears in this generality (to our knowledge) for the first time in Vojta's paper \cite{VojtaLangExc}, and it is also studied in \cite{JXie}. It is intimately related to the notion of ``arithmetic hyperbolicity''   \cite{JAut, JLalg}; see Section \ref{section:ps_mordell} for a discussion. 

In the non-proper case, it is   natural to study  integral points rather than rational points. Vojta noticed in \cite{VojtaLangExc} that, in fact, it is more natural to study ``near-integral points''. Below we make this more precise.

\begin{definition}\label{def:nip}
Let $X\to S$ be a morphism of schemes with $S$  integral.  We define $X(S)^{(1)}$ to be the set of  $P$ in $X(K(S))$ such that, for every point $s$ in $S$ of codimension one, the point $P$ lies in the image of $X(\mathcal{O}_{S,s}) \to X(K(S))$.
\end{definition}

Vojta refers to the points in $X(S)^{(1)}$ as ``near-integral'' $S$-points.  We point out that on an affine variety, there is no difference between the finiteness of integral points and ``near-integral'' points; see Section \ref{section:ps_mordell}.

\begin{definition}[Quasi-projective Mordellic varieties]
A variety  $X$ over $k$ is \emph{Mordellic over $k$} if, for every $\ZZ$-finitely generated subring $A\subset k$ and every model  $\mathcal{X}$ for $X$ over $A$,  the set  $\mathcal{X}(A)^{(1)}$ of near-integral $A$-points is finite.
\end{definition}

The study of near-integral points might seem unnatural at first. To convince the reader that this notion is slightly more natural than the notion of integral point, we include the following   remark.
 
 \begin{remark}[Why ``near-integral'' points?]  Consider a proper scheme $\mathcal{X}$ over $\mathbb{Z}$ with generic fibre $X:=\mathcal{X}_{\mathbb{Q}}$. Let $K$ be a finitely generated field of characteristic zero and let $A\subset K$ be a regular $\ZZ$-finitely generated subring. Then, the set of $K$-rational points $X(K)$  \textbf{equals} the set of near-integral $A$-points of $\mathcal{X}$. On the other hand,  if $K$ has transcendence degree at least one over $\QQ$, then it is not necessarily true that every $K$-point of $X$ is an $A$-point of $\mathcal{X}$. Thus,   studying $K$-rational points on the  proper variety $X$  over $\QQ$ is equivalent to studying near-integral points of the proper scheme $\mathcal{X}$ over $\ZZ$.
 \end{remark}

With this definition at hand, we are able to state Faltings's finiteness theorem for abelian varieties over number rings as a statement about the Mordellicity of the appropriate moduli space. The analogous statement on its Brody hyperbolicity is  Example \ref{exa:bd_Ag}.

\begin{theorem}[Faltings, Shafarevich's conjecture for principally polarized abelian varieties]\label{thm:shaf_Faltings} Let $k$ be an algebraically closed field of characteristic zero.
Let $g\geq 1$ and let $N\geq 3$ be integers. Then, 
 the (fine) moduli space $\mathcal{A}_{g,k}^{[N]}$ of $g$-dimensional  principally polarized abelian varieties with level $N$ structure is a smooth quasi-projective Mordellic variety over $k$. 
\end{theorem}

 Example \ref{exa:bd_Ag}  and Theorem \ref{thm:shaf_Faltings} suggest that there might also be an analogue of Lang-Vojta's conjecture for quasi-projective schemes. It seems reasonable to suspect that an affine variety over $k$ is Mordellic over $k$ if and only if it is strongly-Brody hyperbolic over $k$; see for example \cite{JLevin} for a discussion of Lang's conjectures in the affine case. However,  stating a reasonable conjecture for quasi-projective varieties   requires some care, and would take us astray from our current objective.    We refer the interested reader to articles of Ascher-Turchet and Campana  in this book \cite{AscherTurchetBook, CampanaBook} for a related discussion, and the  book  by Vojta \cite{Vojta3}.
 
 \begin{remark}[From Shafarevich to Mordell]
Let us briefly explain how Faltings shows that Theorem \ref{thm:shaf_Faltings} implies Faltings's finiteness theorem for curves (Theorem \ref{thm:faltings_curves}). Let $X$ be a smooth projective connected curve of genus at least two over $k$. By a construction of Kodaira \cite{MartinDeschamps}, there is a finite \'etale morphism $Y\to X$, an integer $g\geq 1$, and a non-constant morphism $Y\to \mathcal{A}_{g,k}^{[3]}$. Since $\mathcal{A}_{g,k}^{[3]}$ is  Mordellic    over $k$ and $Y\to \mathcal{A}_{g,k}^{[3]}$ has finite fibres, it follows that $Y$ is  Mordellic    over $k$.  As Mordellicity descends along finite \'etale morphisms (Remark \ref{remark:cw_mordell}), we conclude that $X$ is Mordellic, as required.
 \end{remark}

 \section{Groupless varieties}\label{section:groupless}
 To study Lang-Vojta's conjectures, it is natural to study varieties which do  not  ``contain'' any algebraic groups. Indeed, as we have explained in Remark \ref{remark:abvar_Brody} (resp. Remark  \ref{remark:ht}), a Brody hyperbolic variety (resp. a Mordellic variety) does not admit any non-trivial morphisms from an abelian variety. For projective varieties, it turns out that this is equivalent  to not admitting a non-constant map from any connected algebraic group (see Lemma \ref{lem:why_groupless} below). 
 
As before,   we let $k$ be an algebraically closed field of characteristic zero.
We start with the following definition.
\begin{definition}
A variety $X$ over $k$ is \emph{groupless} if every morphism  $\mathbb{G}_{m,k}\to X$ (of varieties over $k$) is constant, and for every every abelian variety $A$ over $k$, every morphism $A\to X$ is constant.
\end{definition}

\begin{remark}\label{remark:gr}
We claim that , for proper varieties, the notion of grouplessness can be tested on morphisms (or even rational maps) from abelian varieties.
 That is,  
 a proper variety  $X$ over $k$ is groupless if and only if, for every abelian variety  $A$ over $k$, every rational map $A\dashrightarrow X$ is constant. To show this, first note that   a morphism $\mathbb{G}_{m,k}\to X$ extends to a morphism $\mathbb{P}^1_k\to X$ and that $\mathbb{P}^1_k$ is surjected upon by an elliptic curve. Therefore, if every morphism from an abelian variety  is constant, then  $X$ is groupless and has no rational curves.   Now, if $X$ is proper over $k$ and has no rational curves,     every rational map $A\dashrightarrow X$ with $A$ an abelian variety extends to a morphism (see \cite[Lemma~3.5]{JKa}). Thus, if  every morphism $A\to X$ is constant with $A$ an abelian variety, we conclude that every rational map $A\dashrightarrow X $ is constant. This proves the claim. We  also conclude that a proper variety is groupless  if and only if it is ``algebraically hyperbolic'' in Lang's sense \cite[p.~176]{Lang2}. 
 \end{remark}
 
 \begin{remark}[Lang's algebraic exceptional set]
 For $X$ a proper variety over $k$, Lang defines the \emph{algebraic exceptional set $\mathrm{Exc}_{alg}(X)$ of $X$} to be the union of all non-constant rational images of abelian varieties in $X$. With Lang's terminology at hand, as is explained in Remark \ref{remark:gr}, a proper variety $X$ over $k$ is groupless over $k$ if and only if $\mathrm{Exc}_{alg}(X)$ is empty.
 
 \end{remark}

Let us   clear up why we refer to this property as groupless.

\begin{lemma}[Why call this groupless?]\label{lem:why_groupless}
A variety $X$ over $k$ is groupless if and only if for all   finite type connected group schemes $G$ over $k$, every morphism $G\to X$ is constant.
\end{lemma}
\begin{proof}
This follows from Chevalley's structure theorem for   algebraic groups over the algebraically closed field $k$ of characteristic zero. A detailed proof is given in \cite[Lemma~2.5]{JKa}.
\end{proof}

 The  notion of grouplessness   is well-studied, and sometimes referred to as ``algebraic hyperbolicity'' or ``algebraic Lang hyperbolicity''; see   \cite{ShangZhang}, \cite[page~176]{Lang2}, \cite[Remark~3.2.24]{KobayashiBook}, or \cite[Definition~3.4]{KovacsSubs}. We will only  use the term ``algebraically hyperbolic'' for the notion introduced by Demailly in \cite{Demailly} (see also \cite{vBJK, JKa,  JXie}).  The term ``groupless'' was first used in  \cite[Definition~2.1]{JKa} and \cite[Definition~3.1]{JVez}.

\begin{example} \label{ex} 
A zero-dimensional variety is groupless. Note that
$\mathbb{P}^1_k$, $\mathbb{A}^1_k$, $\mathbb{A}^1_k\setminus \{0\}$ and smooth proper  genus one curves over $k$  are not groupless.  
\end{example}

Much like Brody hyperbolicity and Mordellicity, grouplessness descends along finite \'etale morphisms. We include a sketch of the  proof of this  simple fact.

\begin{lemma}[Descending grouplessness]\label{lem:cw_groupless}
Let $X\to Y$ be a  finite \'etale morphism of varieties over $k$. Then $X$ is groupless over $k$ if and only if $Y$ is groupless over $k$. 
\end{lemma}
\begin{proof} If $Y$ is groupless, then  $X$ is obviously groupless. Therefore, to prove the lemma, we may assume that $X$ is groupless.  Let $G$ be $\mathbb{G}_{m,k}$ or an abelian variety over $k$. Let $G\to Y $ be a morphism. Consider the  pull-back $G':=G\times_Y X$ of $G\to Y$ along $X\to Y$. Then, as $k$ is algebraically closed and of characteristic zero, each connected component of $G'$ is (or: can be endowed with the structure of)   an algebraic group isomorphic to $\mathbb{G}_{m,k} $ or an abelian variety over $k$. Therefore, the morphism $G'\to X$ is   constant. This implies that $G\to Y$ is constant.
\end{proof}

We include an elementary proof of the fact that the classification of one-dimensional groupless varieties is the same as that of one-dimensional   Brody hyperbolic  curves.
\begin{lemma} 
A smooth quasi-projective connected curve $X$ over $k$ is groupless over $k$ if and only if $X$ is not isomorphic to $\mathbb{P}^1_k$, $\mathbb{A}^1_k$, $\mathbb{A}^1_k\setminus\{0\}$, nor a smooth proper connected curve of genus one over $k$.
\end{lemma}
\begin{proof} If $X$ is groupless, then $X$ is not isomorphic to $\mathbb{P}^1_k$, $\mathbb{A}^1_k$, $\mathbb{A}^1_k\setminus\{0\}$, nor a smooth proper connected curve of genus one over $k$; see Example \ref{ex}.  Thus to prove the lemma,, we may (and do) assume that $X$ is not isomorphic to either of these curves.  
 Let $Y\to X$ be a finite \'etale cover of $X$ such that the smooth projective model $\overline{Y}$ of $Y$ is of genus at least two. (It is clear that such a cover exists when $X = \mathbb{G}_{m,k}\setminus \{1\}$ or $X = E\setminus \{0\}$ with $E$ an elliptic curve over $k$. This  is enough to conclude that such a cover always exists.) By Lemma \ref{lem:cw_groupless}, the variety $X$ is groupless if and only if $ {Y}$ is groupless. Thus, it suffices to show that $\overline{Y}$ is groupless. To do so,  assume that we have a morphism $\mathbb{G}_{m,k}\to \overline{Y}$. By Riemann-Hurwitz, this   morphism is constant, as $\overline{Y}$ has genus at least two. Now, let $A$ be an abelian variety over $k$ and let  $A\to \overline{Y}$ be a morphism.  To show that this morphism is constant, we compose $A\to \overline{Y}$ with the Jacobian map $\overline{Y}\to \mathrm{Jac}(\overline{Y})$ (after choosing some point on $\overline{Y}$). If the morphism $A\to \overline{Y}$ is non-constant, then it is surjective. Since a morphism of abelian varieties is a homomorphism (up to translation of the origin), this induces a group structure on the genus $>1$ curve  $\overline{Y}$. However, as the automorphism group of (the positive-dimensional variety) $\overline{Y}$ is finite,  the curve $\overline{Y}$ can not be endowed with the structure of an algebraic group. This shows that $A\to \overline{Y}$ is constant, and concludes the proof.
 \end{proof}

Bloch--Ochiai--Kawatama's theorem (Theorem \ref{thm:bok}) and Faltings's analogous theorem  for rational points on closed subvarieties of abelian varieties (Theorem \ref{thm:faltings}) characterize ``hyperbolic'' subvarieties of abelian varieties. It turns out that this characterization also holds for groupless varieties, as we explain now.

 If $X$ is a closed subvariety  of an abelian variety $A$ over $k$, we define  the special locus $\mathrm{Sp}(X)$ of $X$ to be  the union of the translates of positive-dimensional abelian subvarieties of $A$ contained in $X$.
 
 \begin{lemma} \label{lem:special_locus}
 Let $X$ be a  closed subvariety of an abelian variety $A$ over $k$. Then $X$ is groupless  over $k$ if and only if $\mathrm{Sp}(X)$ is empty.
 \end{lemma}
 \begin{proof} Clearly, if $X$ is groupless over $k$, then $X$ does not contain the translate of a positive-dimensional abelian subvariety of $A$, so that $\mathrm{Sp}(X)$ is empty.  Conversely, assume that $X$ does not contain the translate of a non-zero abelian subvariety of $A$. Let us show that $X$ is groupless.
 Since the Albanese variety of $\mathbb{P}^1_k$ is trivial, any map $\mathbb{G}_{m,k}\to X$ is constant. Thus, to conclude the proof, we   have to show that all morphisms $A'\to X$ are constant, where $A'$ is an abelian variety over $k$. To do so, note that the image of $A'\to X $  in $A$     is the translate of an  abelian subvariety of $A$, as morphisms of abelian varieties are homomorphisms up to translation. This means that the image of $A'\to X$ is the translate of an abelian subvariety, hence a point (by our assumption).
 \end{proof}

 \begin{remark}\label{remark:groupless}
 Let $A$ be a simple abelian surface. Let $X = A\setminus \{0\}$. Then  $X$ is groupless.  This remark might seem misplaced, but it shows that ``grouplessness'' as defined above   does not capture the non-hyperbolicity of a quasi-projective variety. The ``correct'' definition in the quasi-projective case  is discussed in Section \ref{section:ps_groupless} (and is also discussed   in
  \cite{JXie, VojtaLangExc}).
 \end{remark}
 
 Although grouplessness does not capture the non-hyperbolicity of  quasi-projective varieties (Remark \ref{remark:groupless}), Lang conjectured that grouplessness is equivalent to  being Mordellic and to being Brody hyperbolic (up to choosing a model over $\CC$) for \emph{projective} varieties.  This brings us to the second form of Lang-Vojta's conjecture.
\begin{conjecture}[Weak Lang-Vojta, II]\label{conj:II}
Let $X$ be a projective variety over $k$.
Then the following are equivalent.
\begin{enumerate}
\item The projective variety $X$ is   Mordellic over $k$.
\item The variety $X$ is strongly-Brody hyperbolic over $k$.
\item The variety $X$ is groupless over $k$.
\end{enumerate}
\end{conjecture}

\section{Varieties of general type}\label{section:general_type}
  In this section we discuss the role of varieties of general type in Lang-Vojta's conjecture.
  Recall that a line bundle $L$ on a smooth projective variety $S$ over $k$ is big if there is an ample line bundle $A$ and an effective divisor $D$ such that $L\cong A\otimes \mathcal{O}_S(D)$; see \cite{Lazzie1, Lazzie2}.
We   follow standard terminology and say that an integral proper variety $X$ over   $k$ is of general type if it has a desingularization $X'\to X$ with $X'$ a smooth projective integral variety  over $k$   such that the canonical bundle $\omega_{X'/k}$ is a big line bundle. For example, if $\omega_{X'/k}$ is ample, then it is big. Moreover, we will say that a proper variety $X$ over a field $k$ is of general type if, for every irreducible component $Y$ of $X$, the  reduced closed subscheme $Y_\mathrm{red}$ is of general type.  

Varieties of general type are well-studied; see \cite{Lazzie1, Lazzie2}. For the sake of clarity, we briefly collect some statements. Our aim is to emphasize the similarities with the properties presented in the earlier sections.

For example, much like Brody hyperbolicity, Mordellicity, and grouplessness, the property of being of general type descends along finite \'etale morphisms.
That is, if $X\to Y $ is a finite \'etale morphism of proper schemes over $k$, then $X$ is of general type if and only if $Y$ is of general type.
Moreover, 
a  simple  computation of the degree of the canonical bundle of a curve implies that, if $X$ is  a smooth projective connected curve over $k$, then $X$ is of general type if and only if $\mathrm{genus}(X)\geq 2$.

Kawamata and Ueno classified which closed subvarieties of an abelian variety are of general type. To state their result, for $A$   an abelian variety over $k$ and $X $ a closed subvariety of $A$, recall that the special locus $\mathrm{Sp}(X)$ of $X$ is the union of translates of positive-dimensional abelian subvarieties of $A$ contained in $X$. Note that Bloch--Ochiai--Kawamata's theorem (Theorem \ref{thm:bok}) can be stated as saying that a closed subvariety $X$ of an abelian variety $A$ over $\CC$ is Brody hyperbolic if and only if $\mathrm{Sp}(X)$ is empty. Similarly, Faltings's theorem (Theorem \ref{thm:faltings}) can be stated as saying that a closed subvariety of an abelian variety $A$ over $k$ is Mordellic if and only if $\mathrm{Sp}(X)$ is empty. The latter is also equivalent to saying that $X$ is groupless over $k$ by Lemma \ref{lem:special_locus}.   The theorem of Kawamata-Ueno  now reads as follows.

\begin{theorem}[Kawamata-Ueno]\label{thm:ku} Let $A$ be an abelian variety and let $X$ be a closed subvariety of $A$. Then $\mathrm{Sp}(X)$ is a closed subset of $X$, and  $X$ is of general type if and only if $\mathrm{Sp}(X) \neq X$.
\end{theorem}

Note that being of general type and being groupless are not equivalent. This is not a surprise, as the notion of general type is a birational invariant whereas the blow-up of a smooth groupless surface along a point is no longer groupless. The conjectural relation between varieties of general type and  the three notions (Brody hyperbolicity, Mordellicity, and grouplessness) introduced above is as follows.

\begin{conjecture}[Weak Lang-Vojta, III]\label{conj:III}
Let $X$ be a projective variety over $k$.
Then the following are equivalent.
\begin{enumerate}
\item The projective variety $X$ is Mordellic   over $k$.
\item The variety $X$ is strongly-Brody hyperbolic over $k$.
\item Every integral subvariety of $X$ is of general type.
\item The variety $X$ is groupless over $k$.
\end{enumerate}
\end{conjecture}

 Note that the notion of general type is a birational invariant, but hyperbolicity is not. What should (conjecturally) correspond to  being of general type? The highly optimistic conjectural answer is that being of general type should correspond to being ``pseudo''-Brody hyperbolic, ``pseudo''-Mordellic, and ``pseudo''-groupless. The definitions of these notions are essentially the same as given above, the only difference being that one has to allow for an ``exceptional locus''. In the following sections we will make this more precise.
 
 \section{Pseudo-grouplessness}\label{section:ps_groupless}
Let $k$ be an algebraically closed field of characteristic zero.
Roughly speaking, a projective variety $X$ over $k$ is groupless if it admits no non-trivial morphisms from a connected algebraic group. Conjecturally,   a projective variety $X$ over $k$ is groupless if and only if  every subvariety of $X$ is of general type. To see what should correspond to being of general type, we will require the more general notion of pseudo-grouplessness.

 \begin{definition}   Let $X$ be a variety over $k$ and let $\Delta\subset X$ be a closed subset.
 We say that  $X$  is \emph{groupless modulo $\Delta$ (over $k$)} if,  for every  finite type connected group scheme $G$ over $k$ and every dense open subscheme $U\subset G$ with $\mathrm{codim}(G\setminus U)\geq 2$, every non-constant morphism $U\to X$ factors over $\Delta$. 
\end{definition}

Hyperbolicity modulo a subset was first introduced by Kiernan--Kobayashi \cite{KiernanKobayashi}, and is   thoroughly studied in Kobayashi's book \cite{KobayashiBook}. As we will see below, it is quite natural to extend the study of hyperbolic varieties to the study of varieties which are hyperbolic modulo a subset.

 For proper schemes, the notion of  ``groupless modulo the empty set'' coincides with the notion of grouplessness introduced before (and studied in \cite{JAut, JKa, JVez}). For the reader's convenience, we include a detailed proof of this.
 
\begin{lemma} Let $X$ be a proper scheme over $k$. Then the following are equivalent.
\begin{enumerate}
\item The scheme $X$ is groupless modulo the empty subscheme $\emptyset$ over $k$.
\item The scheme $X$ is groupless.
\item For every  finite type connected group scheme $G$ over $k$ and every dense open subscheme $V\subset G$, every morphism $V\to X$ is constant.
\end{enumerate}
\end{lemma}
\begin{proof}
It is clear  that $(1)$ implies $(2)$. To show that $(2)$ implies $(3)$, let $G$ be a finite type connected group scheme over $k$, let $V\subset G$ be a  dense open subscheme, and let $f:V\to X$ be a morphism of schemes over $k$. Then, as $X$ is proper over $k$, there is an open subscheme $U\subset G$ containing $V$ with $\mathrm{codim}(G\setminus U)\geq 2$ such that  the morphism $f:V\to X$ extends to a morphism $f':U\to X$. Since $X$ is groupless and proper, it does not  contain any rational curves. Therefore, as the variety underlying $G$ is smooth over $k$ \cite[Tag~047N]{stacks-project}, it follows from \cite[Lemma~3.5]{JKa} (see also\cite[Corollary~1.44]{Debarrebook1}) that the morphism $f':U\to X$ extends (uniquely) to  a morphism $f'':G\to X$. Since $X$ is groupless, the morphism $f''$ is constant. This implies that $f$ is constant.  Finally, it is clear (from the definitions) that $(3)$ implies $(1)$.
\end{proof}

\begin{definition}
 A variety $X$ is \emph{pseudo-groupless (over $k$)} if there is a proper closed subset $\Delta\subsetneq X$ such that $X$ is groupless modulo $\Delta$.
\end{definition}

 The word ``pseudo'' in this definition refers to the fact that the non-hyperbolicity of the variety is concentrated in a proper closed subset. Note that a variety $X$ is pseudo-groupless if and only if every irreducible component of $X$ is pseudo-groupless.
 
 \begin{example}\label{exa:blow_up} Let $C$ be smooth projective connected curve of genus at least two and
let $X$ be the blow-up of $C\times C$ in a point. Then $X$ is not groupless. However, its ``non-grouplessness'' is contained in the exceptional locus $\Delta$ of the blow-up $X\to C\times C$. Thus, as $X$ is groupless modulo $\Delta$, it follows that $X$ is pseudo-groupless.
\end{example}

Let us briefly say that an open subset $U$ of an integral variety $V$ is \emph{big} if $\mathrm{codim}(V\setminus U)$ is at least two. Now, 
 the reader might wonder why we test pseudo-grouplessness  on maps whose domain is a  big open subset of some algebraic group.  The example to keep in mind here is the blow-up of a simple  abelian surface in its origin.   In fact, as we  test   pseudo-grouplessness on big open subsets of abelian varieties (and not merely maps from abelian varieties),  such blow-ups are \emph{not} pseudo-groupless.    Also, roughly speaking, one should consider big open subsets of abelian varieties as far as possible from being hyperbolic, in any sense of the word ``hyperbolic''.  For example, much like how abelian varieties admit a dense entire curve (Remark \ref{remark:abvar_Brody}),  a big open subset of an abelian variety admits a dense entire curve. This is proven using Sard's theorem in \cite{VojtaLangExc}. Thus, big open subsets of abelian varieties are also as far as possible from being Brody hyperbolic.

We now show that the statement of Lemma \ref{lem:cw_groupless} also holds in the ``pseudo'' setting, i.e., we  show that pseudo-grouplessness descends along finite \'etale morphisms..    As we have mentioned before, this descent property also holds for general type varieties.

 \begin{lemma}\label{lemma:cw_for_groupless}
 Let $f:X\to Y $ be a finite \'etale morphism of varieties over $k$. Then $X$ is pseudo-groupless over $k$ if and only if $Y$ is pseudo-groupless over $k$.
 \end{lemma}
 \begin{proof} We adapt the  arguments in the proof of  \cite[Proposition~2.13]{JVez}.  First, if $Y$ is groupless modulo a proper closed subset $\Delta_Y\subset Y$, then clearly $X$ is groupless modulo the proper closed subset $f^{-1}(\Delta_Y)$. Now, assume that $X$ is groupless modulo a proper closed subset $\Delta_X\subsetneq X$.   Let $G$ be a finite type connected (smooth quasi-projective) group scheme over $k$, let $U\subset G$ be a dense open subscheme with $\mathrm{codim}(G\setminus U)\geq 2$ and let $\phi:U\to Y$ be a morphism which does not factor over $f(\Delta_X)$. The pull-back of $G\to Y$ along the finite \'etale morphism $f:X\to Y$ induces a finite \'etale morphism $V:=U\times_Y X\to U$.  Since $U$ is smooth over $k$, by purity of the branch locus \cite[Th\'eor\`eme~X.3.1]{SGA1}, the finite \'etale morphism $V\to U$ extends (uniquely) to a finite \'etale morphism $G'\to G$.  Note that every connected component  $G''$ of $G'$ has the structure of a finite type connected group scheme over $k$ (and with this structure the morphism   $G''\to G$ is a homomorphism).  Now, since smooth morphisms are codimension-preserving, we see that $\mathrm{codim}(G''\setminus V) \geq 2$.  As the morphism $V\to X$ does not  factor over $f^{-1}(f(\Delta_X))$, it does not factor over $\Delta_X$, and is thus constant (as $X$ is groupless modulo $\Delta_X$). This implies that the morphism $U\to Y$ is constant, as required.
 \end{proof}

 \begin{remark} [Birational invariance] \label{lem:birational_invariance}
Let $X$ and $Y$ be proper schemes over $k$. Assume that $X$ is birational to $Y$. Then $X$ is pseudo-groupless over $k$ if and only if $Y$ is pseudo-groupless over $k$. This is proven in \cite{JXie}. Thus, as pseudo-grouplessness is a birational invariant amongst proper varieties, this notion   is more natural to study from a birational perspective than grouplessness.
 \end{remark} 
 
 \begin{remark}\label{remark:non_uni} Contrary to a hyperbolic proper variety, 
 a proper pseudo-groupless variety could have rational curves. For example, the blow-up  of the product of two smooth curves of genus two in a point (as in Example \ref{exa:blow_up}) contains precisely one rational curve.  However, a pseudo-groupless proper variety is not  covered by rational curves, i.e., it is non-uniruled, as all rational curves are contained in a proper closed subset (by definition). 
 \end{remark}
 
  \begin{remark}\label{remark:properness}
 Let $X$ be a proper scheme over $k$ and let $\Delta\subset X$ be a closed subset. It follows from the valuative criterion of properness that  $X$ is groupless modulo $\Delta$ if and only if, for every finite type connected group scheme $G$ over $k$ and every dense open subscheme $U\subset G$, any non-constant morphism $U\to X$ factors over $\Delta$. 
 \end{remark}

Recall that Lemma \ref{lem:why_groupless} says that the grouplessness  of a proper variety entails that there are no non-constant morphisms from \emph{any} connected algebraic group.
One of the main results of \cite{JXie} is the analogue of Lemma \ref{lem:why_groupless} for pseudo-groupless varieties.   The proof of this result (see Theorem \ref{lem:groupless2} below) relies on the structure theory of algebraic groups.

\begin{theorem}\label{lem:groupless2}
If $X$ is a proper scheme $X$ over $k$ and $\Delta$ is a closed subset of $X$, then $X$ is groupless modulo $\Delta$ over $k$ if and only if, for every abelian variety $A$ over $k$ and every open subscheme $U\subset A$ with $\mathrm{codim}(A\setminus U)\geq 2$, every non-constant morphism of varieties $U\to X$ factors over $\Delta$.  
\end{theorem}

 Theorem \ref{lem:groupless2} says that the pseudo-grouplessness of a proper variety can be tested on morphisms from big open subsets of abelian varieties (or on rational maps from abelian varieties).  A similar, but different, statement holds for affine varieties. Indeed, 
 if $X$ is an affine variety over $k$, then $X$ is groupless modulo $\Delta\subset X$ if and only if every non-constant morphism $\mathbb{G}_{m,k}\to X$ factors over $\Delta$.

Lang conjectured  that a projective variety is pseudo-groupless if and only if it is of general type. Note that, by the birational invariance of these two notions, this conjecture can be reduced to the case of smooth projective varieties by Hironaka's resolution of singularities.

  \begin{conjecture}[Strong Lang-Vojta, I]\label{conj:S1}
  Let $X$ be a projective variety over $k$. Then   $X$ is pseudo-groupless over $k$ if and only if $X$ is of general type over $k$.
  \end{conjecture}
  
  Note that this conjecture predicts more than the equivalence of $(3)$ and $(4)$ in Conjecture \ref{conj:III}. Also, even though it is stated for projective varieties, one could as well formulate the conjecture for proper varieties (or even proper algebraic spaces). The resulting ``more general'' conjecture actually follows from the above conjecture.
  
 \begin{example}\label{example:groupless}
  By Kawamata-Ueno's theorem (Theorem \ref{thm:ku}) and Lemma \ref{lem:special_locus}, the Strong Lang-Vojta conjecture holds for closed subvarieties of abelian varieties.
 \end{example}

 \begin{remark}\label{remark:surfaces}
 If $X$ is a proper pseudo-groupless surface, then $X$ is of general type (see \cite{JXie} for a proof). For higher-dimensional varieties, Conjecture \ref{conj:S1} predicts a similar statement, but this is not even known  for threefolds. However,  assuming the Abundance Conjecture   and certain conjectures on Calabi-Yau varieties, one can show that every proper pseudo-groupless variety is of general type (i.e.,  $(1)\implies (2)$ in Conjecture \ref{conj:S1}).  
 Regarding the implication $(2)\implies (1)$, not much is known beyond the one-dimensional case. For example, if $X$ is a proper surface of general type, then Conjecture \ref{conj:S1} implies that there should be a proper closed subset $\Delta\subset X$ such that every rational curve $C\subset X$ is contained in $\Delta$. Such statements are   known to hold for certain surfaces of general type  by the work of Bogomolov  and McQuillan; see  \cite{DeschampsBogomolov, McQuillan}.
 \end{remark}

 \section{Pseudo-Mordellicity and pseudo-arithmetic hyperbolicity}\label{section:ps_mordell}
 
 In the previous section, we introduced pseudo-grouplessness  and stated Lang-Vojta's conjecture that a projective variety is of general type if and only if it is pseudo-groupless. In this section, we explain what the ``pseudo'' analogue is of the notion of Mordellicity, and explain Lang-Vojta's conjecture that a projective variety is  of general type
 if and only if it is pseudo-Mordellic.

\subsection{Pseudo-arithmetic hyperbolicity}
As we have said before, Lang coined the term ``Mordellic''. We will now introduce the related (and a priori different) notion of   arithmetic hyperbolicity      (as defined in \cite{JAut, JLitt, JLalg}); see also \cite[\S 2]{UllmoShimura}, and \cite{Autissier1, Autissier2}. In Section \ref{section:Mordell} we ignored   that the extension of the notion of Mordellicity over $\Qbar$ to  arbitrary algebraically closed fields can actually be done in two a priori different ways. We discuss both   notions now and give them \emph{different} names. We refer the reader to Section \ref{section:Mordell} for our conventions regarding models of varieties, and we continue to let $k$ denote an algebraically closed field of characteristic zero.

\begin{definition}
 Let $X$ be a  variety  over $k$ and let $\Delta$ be a closed subset of $X$. We say that $X$ is \emph{arithmetically hyperbolic modulo $\Delta$ over $k$} if, for every $\ZZ$-finitely generated subring $A$ and every model $\mathcal{X}$ for $X$ over $A$, we have that every positive-dimensional irreducible component of the Zariski closure of $\mathcal{X}(A)$ in $X$ is contained in $\Delta$.
 \end{definition}

 \begin{definition}
 A variety $X$ over $k$ is \emph{pseudo-arithmetically hyperbolic over $k$} if there is a proper closed subset $\Delta\subset X$ such that $X$ is arithmetically hyperbolic modulo $\Delta$ over $k$.
 \end{definition}

  \begin{remark}
 A variety $X$ over $k$ is arithmetically hyperbolic over $k$ (as defined in \cite{JAut} and \cite[\S 4]{JLalg}) if and only if $X$ is arithmetically hyperbolic over $k$ modulo the empty subscheme.  
 \end{remark}

 \begin{lemma}[Independence of model]  Let $X$ be a variety  over $k$ and let $\Delta$ be a closed subset of $k$. Then the following are equivalent.
 \begin{enumerate}
 \item The finite type scheme $X$ over $k$ is arithmetically hyperbolic modulo $\Delta$.
 \item There is a $\ZZ$-finitely generated subring $A\subset k$, there is a model $\mathcal{X}$ for $X$ over $A$, and there is a model $\mathcal{D}\subset \mathcal{X}$ for $\Delta\subset X$ over $A$ such that, for every $\ZZ$-finitely generated subring $ B\subset k$ containing $A$, the set \[\mathcal{X}(B)\setminus \mathcal{D}(B) \] is finite.
 \end{enumerate}
 \end{lemma}
 \begin{proof}
 This follows from standard spreading out arguments. These type of arguments are used in  \cite{JLalg} to prove more general statements in which the objects are algebraic stacks. 
 \end{proof}

 \begin{remark}
We unravel what the notion of arithmetic hyperbolicity modulo $\Delta$ entails for affine varieties. To do so, let $X$ be an affine variety over $k$, and let $\Delta$ be a proper closed subset of $X$. Choose the following data. \begin{itemize}
\item  integers $n, \delta, m\geq 1$; 
\item  polynomials $f_1,\ldots, f_n \in k[x_1,\ldots,x_m]$;
\item  polynomials $d_1,\ldots, d_\delta \in k[x_1,\ldots,x_m]$;  
\item  an isomorphism $$X\cong \Spec(k[x_1,\ldots,x_m]/(f_1,\ldots,f_n));$$  
\item an isomorphism 
\[
\Delta \cong   \Spec(k[x_1,\ldots,x_m]/(d_1,\ldots,d_\delta)).
\]
\end{itemize} Let $A_0$ be the $\ZZ$-finitely generated subring of $k$ generated by the (finitely many) coefficients of the polynomials $f_1,\ldots, f_n, d_1,\ldots, d_\delta$.  
Now, the following statements are equivalent.
\begin{enumerate}
\item The variety $X$ is arithmetically hyperbolic modulo $\Delta$ over $k$.
\item For every   $\ZZ$-finitely generated subring $A\subset k$ containing $A_0$, the set
\[
\{a\in A^m \ | \ f_1(a) = \ldots = f_n(a) =0\}\setminus \{ a \in A^m \ | \ d_1(a) = \ldots = d_\delta(a) =0 \}
\] is finite.
\end{enumerate}
 Thus, roughly speaking, one could say that an algebraic variety over $k$ is arithmetically hyperbolic modulo $\Delta$ over $k$ if ``$X$ minus $\Delta$''  has only finitely many $A$-valued points, for any choice of finitely generated subring $A\subset k$. 
 \end{remark}
 
 \subsection{Pseudo-Mordellicity}
 
 The reader might have  noticed a possibly confusing change in terminology. Why do we not refer to the above notion as being ``Mordellic modulo $\Delta$''? The precise reason brings us to a subtle point in the study of integral points valued in higher-dimensional rings (contrary to those valued in $\OO_{K,S}$ with $S$ a finite set of finite places of a number field $K$). To explain this subtle point, let us first define what it means to be pseudo-Mordellic. For this definition, we will require the notion of ``near-integral'' point (Definition \ref{def:nip}).

\begin{definition}
 Let $X$ be a variety over $k$ and let $\Delta$ be a closed subset of $X$. We say that $X$ is \emph{Mordellic   modulo $\Delta$ over $k$} if, for every $\ZZ$-finitely generated subring $A$ and every model $\mathcal{X}$ for $X$ over $A$, we have that every positive-dimensional irreducible component of the Zariski closure of $\mathcal{X}(A)^{(1)}$ in $X$ is contained in $\Delta$, where $\mathcal{X}(A)^{(1)}$  is defined in Definition \ref{def:nip}.
 \end{definition}

 \begin{remark} Let $X$ be a proper scheme over $k$ and let $\Delta$ be a closed subset of $X$. 
Then, by the valuative criterion of properness, the proper scheme  $X$  is Mordellic modulo $\Delta$ if, for every finitely generated subfield $K\subset k$ and every proper model $\mathcal{X}$ over $K$, the set $\mathcal{X}(K)\setminus \Delta$ is finite.
 \end{remark}
 
 \begin{definition}
 A variety $X$ over $k$ is \emph{pseudo-Mordellic   over $k$} if there is a proper closed subset $\Delta\subset X$ such that $X$ is Mordellic    modulo $\Delta$ over $k$.
 \end{definition}
 
 Note that $X$ is Mordellic over $k$ (as defined in Section \ref{section:Mordell}) if and only if $X$ is Mordellic modulo the empty subset.   It is also clear from the definitions that, if $X$ is Mordellic modulo $\Delta$ over $k$, then $X$ is arithmetically hyperbolic modulo $\Delta$ over $k$. In particular, a pseudo-Mordellic variety is pseudo-arithmetically hyperbolic and a Mordellic variety is arithmetically hyperbolic. Indeed,  roughly speaking, to say that a variety is arithmetically hyperbolic 
  is to say that any set of integral points on it is finite, and to say that a variety is Mordellic is to say that any set of ``near-integral'' points on it is finite. The latter sets are a priori bigger.  However, there is no difference between these two sets when $k=\Qbar$. That is, a variety $X$ over $\Qbar$ is arithmetically hyperbolic modulo $\Delta$ if and only if it is Mordellic modulo $\Delta$ over $\Qbar$.

 Following the exposition in the previous sections, let us   prove the fact that pseudo-arithmetic hyperbolicity (resp. pseudo-Mordellicity) descends along finite \'etale morphisms of varieties.

\begin{theorem} [Chevalley-Weil]\label{thm:cw}
Let $f:X\to Y$ be a finite \'etale surjective morphism of varieties over $k$. Let $\Delta\subset X$ be a closed subset. If $X$ is Mordellic   modulo $\Delta$ over $k$ (resp. arithmetically hyperbolic modulo $\Delta$ over $k$), then $Y$ is Mordellic   modulo $f(\Delta)$ over $k$ (resp. arithmetically hyperbolic modulo $f(\Delta)$ over $k$).  
\end{theorem}
\begin{proof}  We assume that $X$ is Mordellic modulo $\Delta$, and show that  $Y$ is Mordellic modulo $f(\Delta)$. (The statement concerning arithmetic hyperbolicity is proven similarly.)

  Let $A\subset k$ be a regular $\ZZ$-finitely generated subring, let $\mathcal{X}$ be a model for $X$ over $A$, let $\mathcal{Y} $ be a model for $Y$ over $A$,   and let $F:\mathcal{X}\to \mathcal{Y}$ be a finite \'etale surjective morphism such that $F_k = f$. Assume for a contradiction that $Y$ is not Mordellic   modulo $f(\Delta)$. Then, replacing $A$ by a larger regular $\ZZ$-finitely generated subring of $k$ if necessary,   for $i=1, 2, \ldots $, we may choose pairwise distinct elements $a_i$  of $\mathcal{Y}(A)^{(1)}$ whose closure in $Y$ is an irreducible positive-dimensional subvariety $R\subset Y$ such that $R\not\subset f(\Delta)$. For every $i=1,2,\ldots$, choose a dense open subscheme $U_i$ of $\Spec A$ whose complement in $\Spec A$ has codimension at least two and such that $a_i$ defines a morphism $a_i:U_i\to \mathcal{X}$. Consider $V_i:= U_i \times_{\mathcal{Y}, F} \mathcal{X}\to \mathcal{X}$, and note that $V_i\to U_i$ is finite \'etale. By Zariski-Nagata purity of the  branch locus \cite[Th\'eor\`eme~X.3.1]{SGA1},  the morphism $V_i\to U_i$ extends to a finite \'etale morphism $\Spec B_i\to A$.   By Hermite's finiteness theorem, as the degree of $B_i$ over $A$ is bounded by $\deg(f)$,  replacing $a_i$ by an infinite subset if necessary, we may and do assume that $B:=B_1 \cong B_2 \cong B_3 \cong \ldots$. Now, the $b_i:V_i \to \mathcal{X}$ define elements in $\mathcal{X}(B)^{(1)}$. Let $S$ be their closure in $X$. Note that $R\subset S$. In particular, $S\not\subset \Delta$. This contradicts the fact that $X$ is Mordellic modulo $\Delta$. Thus, we conclude that $Y$ is Mordellic modulo $f(\Delta)$.
\end{proof}

\begin{corollary}[Pseudo-Chevalley-Weil]\label{cor:cw2}
Let $f:X\to Y$ be a finite \'etale surjective morphism of finite type separated schemes over $k$. Then $X$ is pseudo-Mordellic over $k$ if and only if $Y$ is pseudo-Mordellic over $k$.
\end{corollary}
\begin{proof}  Since $f:X\to Y$ has finite fibres, the fibres of $f$ are Mordellic over $k$. Therefore,  if $Y$ is pseudo-Mordellic over $k$,   it easily follows that $X$ is pseudo-Mordellic over $k$. Conversely, if $X$ is pseudo-Mordellic over $k$, then it follows from  Theorem \ref{thm:cw}  that $Y$ is pseudo-Mordellic over $k$.
\end{proof}

\begin{corollary}[Pseudo-Chevalley-Weil, II]
Let $f:X\to Y$ be a finite \'etale surjective morphism of finite type separated schemes over $k$. Then $X$ is pseudo-arithmetically hyperbolic over $k$ if and only if $Y$ is pseudo-arithmetically hyperbolic over $k$.
\end{corollary}
\begin{proof}  Similar to the proof of Corollary \ref{cor:cw2}.
\end{proof}

 \begin{remark}[Birational invariance]
 The birational invariance of the notion of pseudo-Mordellicity is essentially built into the definition. Indeed,  the infinitude of the set of near-integral points is preserved under proper birational modifications.  More precisely, 
 let $X$ and $Y$ be proper integral varieties over $k$ which are birational. Then $X$ is pseudo-Mordellic over $k$ if and only if $Y$ is pseudo-Mordellic over $k$. 
 \end{remark}
 
 It is not  clear to us whether the notion of pseudo-arithmetic hyperbolicity over $k$ is a birational invariant for proper varieties over $k$, unless $k=\Qbar$.
 Similarly, it is not so clear to us whether pseudo-arithmetically hyperbolic proper varieties are pseudo-groupless. On the other hand, this is not so hard to prove for pseudo-Mordellic varieties.
 
 \begin{theorem}\label{thm:ar_is_gr}  
 If $X$ is a pseudo-Mordellic  proper variety over $k$, then $X$ is pseudo-groupless over $k$.
 \end{theorem}
 \begin{proof}  
The fact that an arithmetically hyperbolic variety is groupless is proven in \cite[\S 3]{JAut} using Hassett--Tschinkel's theorem on potential density of rational points on an abelian variety over a field $K$ of characteristic zero (Remark \ref{remark:ht}). The statement of the theorem is proven in \cite{JXie} using similar arguments.
 \end{proof}

\begin{remark}
Let $X$ be a proper surface over $k$. If $X$ is pseudo-Mordellic over $k$, then  $X$ is of general type.
To prove this, note that $X$ is pseudo-groupless (Theorem \ref{thm:ar_is_gr}), so that the claim follows from the fact that pseudo-groupless proper surfaces are of general type; see Remark \ref{remark:surfaces}.
\end{remark}
 
Recall that a closed subvariety $X$ of an abelian variety $A$ is groupless modulo its special locus $\mathrm{Sp}(X)$, where $\mathrm{Sp}(X)$ is the union of translates of non-zero abelian subvarieties of $A$ contained in $X$. (We are freely using here Kawamata-Ueno's theorem that $\mathrm{Sp}(X)$ is a closed subset of $X$.) This was proven  in Lemma \ref{lem:special_locus}. In \cite{FaltingsLang} Faltings proved the arithmetic analogue of this statement.
 
 \begin{theorem}[Faltings]\label{thm:faltings_big} Let $A$ be an abelian variety over $k$, and let $X\subset A$ be a closed subvariety. Then   $X$ is Mordellic modulo $\mathrm{Sp}(X)$.
 \end{theorem}

Lang and Vojta conjectured that a projective variety over $\Qbar$ is  pseudo-Mordellic if and only if it is of general type. We propose extending this to arbitrary algebraically closed fields of characteristic zero. As we also expect the notions of pseudo-arithmetic hyperbolicity and pseudo-Mordellicity to coincide, we include this in our version of the Lang-Vojta conjecture.

  \begin{conjecture}[Strong Lang-Vojta, II]\label{conj:S2}
  Let $X$ be a projective variety over $k$. Then the following statements are equivalent.
  \begin{enumerate}
  \item The   variety $X$ is pseudo-Mordellic over $k$.
  \item The   variety $X$ is pseudo-arithmetically hyperbolic over $k$.
  \item  The   variety $X$ is pseudo-groupless over $k$.
  \item The projective variety $X$ is of general type over $k$.
  \end{enumerate}
  
  \end{conjecture}

 This is a good time to collect   examples of arithmetically hyperbolic varieties.

\begin{example}  It follows from Faltings's theorem \cite{FaltingsLang} that  a normal projective connected pseudo-groupless surface $X$ over $k$ with $\mathrm{h}^1(X,\mathcal{O}_X) > 2$ is pseudo-Mordellic. Let us prove this claim.  To do so, let $\Delta\subset X$ be a proper closed subset such that $X$ is groupless modulo $\Delta$. Moreover, let $A$ be the Albanese variety of $X$, let $p:X\to A$ be the canonical map (after choosing some basepoint in $X(k)$), and  let $Y$ be the image of $X$ in $A$. Note that $\dim Y \geq 1$. If $\dim Y = 1$, then the condition on the dimension of $A$ implies that $Y$ is not an elliptic curve. In this case, since $\dim X = 2$ and $\dim Y = 1$, the claim   follows from Faltings's (earlier) finiteness theorem for hyperbolic curves. However, if $\dim Y = 2$, we have to appeal to Faltings's Big Theorem. Indeed, in this case, the morphism $X\to Y$ is generically finite. Let $X\to X'\to Y$ be the Stein factorization of the morphism $X\to Y$, where $X'\to Y$ is a finite morphism with $X'$ normal. Since $X$ and $X'$ are birational, it suffices to show that $X'$ is pseudo-Mordellic (by the birational invariance of pseudo-Mordellicity and pseudo-grouplessness).  Thus, we may and do assume that $X=X'$, so that $X\to A$ is finite.  
If the rational points on $X$ are dense, then they are also dense in $Y$, so that $Y$ is an abelian subvariety of $A$, contradicting our assumption that   $\mathrm{h}^1(X,\mathcal{O}_X) = \dim A > 2$. Thus, the rational points on $X$ are not dense.  In particular, every irreducible component of the closure of a set of rational points on $X$ is a curve of genus $1$ (as $X$ does not admit any curves of genus zero). Since $X$ is pseudo-groupless, these components are contained in $\Delta$.
\end{example}

 \begin{example} 
 Let $X$ be a smooth projective connected curve over $k$,  let $n\geq 1$ be an integer, and let $\Delta$ be a  proper closed subset of $\mathrm{Sym}^n_X$. It follows from Faltings's theorem that  $\mathrm{Sym}^n_X$ is  groupless modulo $\Delta$ over $k$ if and only if $\mathrm{Sym}_X^n$ is arithmetically hyperbolic modulo $\Delta$ over $k$.  
 \end{example}
 
 \begin{example}[Moriwaki]
 Let $X$ be a smooth projective variety over $k$ such that $\Omega^1_X$ is ample and globally generated. Then $X$ is  Mordellic by a theorem of Moriwaki \cite{Moriwaki}. We refer the reader to Ascher-Turchet's chapter in this book for the analogous finiteness result in the logarithmic case \cite{AscherTurchetBook}.
 \end{example}
 
 \begin{example}
For every $\ZZ$-finitely generated normal integral domain $A$ of characteristic zero, the set of $A$-isomorphism classes of smooth sextic surfaces  in $\mathbb{P}^3_A$ is finite; see \cite{JLFano}. This finiteness statement can be reformulated as saying that the moduli stack of smooth sextic surfaces is Mordellic.
 \end{example}

  \begin{example}
  Let $X$ be  a smooth proper hyperkaehler variety over $k$ with Picard number at least three. Then $X$ is not arithmetically hyperbolic; see \cite{JAut}.
  \end{example}
  \subsection{Intermezzo: Arithmetic hyperbolicity and Mordellicity}\label{section:ps_mordell2}
  Let $k$ be an algebraically closed field of characteristic zero. In this section, we show that the (a priori) difference between arithmetic hyperbolicity (modulo some subset) and Mordellicity is quite subtle, as this difference disappears in many well-studied cases.  
  
  The following notion of purity for models over $\ZZ$-finitely generated rings  was first considered in \cite{vBJK} precisely to study the a priori difference between arithmetic hyperbolicity and Mordellicity. 
  
  \begin{definition}[Pure model]
 Let $X$ be a variety over $k$ and let $A\subset k$ be a subring. A model $\mathcal{X}$ for $X$ over $A$ is \emph{pure over $A$} (or: \emph{satisfies the extension property over $A$}) if, for every smooth finite type separated integral scheme $T$ over $A$, every dense open subscheme $U\subset T$ with $T\setminus U$   of codimension at least two in $T$, and every $A$-morphism $f:U\to \mathcal{X}$, there is a (unique) morphism $\overline{f}:T\to \mathcal{X}$   extending the morphism $f$.   (The uniqueness of the extension $\overline{f}$ follows from our convention that a model for $X$ over $A$ is separated.)
  \end{definition}
  
  \begin{remark} Let $X$ be a variety over $k$, and let $A\subset k$ be a subring.
  Let $\mathcal{X}$ be a pure model for $X$ over $A$, and let $ B\subset k$ be a  subring  containing $A$ such that $\Spec B \to \Spec A$ is smooth (hence finite type). Then $\mathcal{X}_B$ is pure over $B$. 
  \end{remark}
  
  \begin{definition} A variety $X$ over $k$ has an \emph{arithmetically-pure model} if there is a $\ZZ$-finitely generated subring $A\subset k$ and a pure model $\mathcal{X}$ for $X$ over $A$.
  \end{definition}

\begin{remark}
Let $X$ be a proper variety over $k$ which has an arithmetically-pure model. Then $X$ has no rational curves. To prove this, assume that $\mathbb{P}^1_k\to X$ is a non-constant (hence finite) morphism, i.e., the proper variety $X$ has a rational curve over $k$. Then, if we let $0$ denote the point $(0:0:1)$ in $\mathbb{P}^2_k$, the composed morphism $\mathbb{P}^2_k\setminus\{0\}\to \mathbb{P}^1_k\to X$ does not extend to a morphism from $\mathbb{P}^2_k$ to $X$. Now,   choose a $\ZZ$-finitely generated subring $A\subset k$ and a   model $\mathcal{X}$ over $A$ such that the morphism $\mathbb{P}^1_k\to X$ descends to a morphism $\mathbb{P}^1_A\to \mathcal{X}$ of $A$-schemes. Define $U=\mathbb{P}^2_A\setminus\{0\}$ and $T = \mathbb{P}^2_A$, where we let $\{0\}$ denote  the image of the section of $\mathbb{P}^2_A\to \Spec A$ induced by $0$ in $\mathbb{P}^2_k$.  Since the morphism $U_k\to \mathcal{X}_k$ does not extend to a morphism $T_k\to X_k$, we see that the morphism $U\to \mathcal{X}$ does not extend to a morphism $T\to \mathcal{X}$, so that $\mathcal{X}$ is not pure.  This shows that a proper variety over $k$ with a rational curve has no arithmetically-pure model.
\end{remark}

\begin{remark} Let $X$ be a proper variety over $k$. A pure model for $X$ over a $\ZZ$-finitely generated subring $A$ of $k$ might have rational curves in every special fibre (of positive characteristic). Examples of such varieties can be constructed as complete subvarieties of the moduli space of principally polarized abelian varieties.
\end{remark}

\begin{remark} Let $X$ be a smooth projective   variety over $k$.
If $\Omega^1_{X/k}$ is ample, then $X$ has an arithmetically-pure model. Indeed, choose  a  $\ZZ$-finitely generated subring $A\subset k$ with $A$ smooth over $\ZZ$ and  a smooth projective model $\mathcal{X}$ for $X$ over $A$ such that $\Omega_{\mathcal{X}/A}$ is ample. Then, the geometric fibres of $\mathcal{X}\to \Spec A$ do not contain any rational curves, so that    \cite[Proposition~6.2]{GLL} implies that $\mathcal{X}$ is a pure model for $X$  over $A$. 
\end{remark}

\begin{remark}
Let $k\subset L$ be an extension of algebraically closed fields of characteristic zero, and let $X$ be a variety over $k$. Then $X$ has an arithmetically-pure model if and only if $X_L$ has an arithmetically-pure model.
\end{remark}
  
  \begin{theorem}\label{thm:ar_hyp_is_mor}
Let $X$ be a variety over $k$ which has an arithmetically-pure model. Let $\Delta \subset X$ be a closed subset. Then $X$ is Mordellic modulo $\Delta$ over $k$ if and only if $X$ is arithmetically hyperbolic modulo $\Delta$ over $k$.
  \end{theorem}
  \begin{proof}We follow the proof of \cite[Theorem~8.10]{vBJK}.
  Suppose that $X$ is arithmetically hyperbolic modulo $\Delta$ over $k$. Let $A\subset k$ be a $\ZZ$-finitely generated subring and let $\mathcal{X}$ be a pure model for $X$ over $A$. It suffices to show that, for every $\ZZ$-finitely generated subring $B\subset k$ containing $A$, the set $\mathcal{X}(B)^{(1)}\setminus \Delta$ is finite. To do so, we may and do assume that $\Spec B\to \Spec A$  is smooth in which case it follows from the definition of a pure model that $\mathcal{X}(B)^{(1)} = \mathcal{X}(B)$. We conclude that 
  \[
  \mathcal{X}(B)^{(1)}\setminus \Delta   =\mathcal{X}(B)\setminus \Delta
  \] is finite. This shows that $X$ is Mordellic modulo $\Delta$ over $k$.
  \end{proof}
  
    \begin{lemma}[Affine varieties]\label{lem:affine}
  Let $X$ be an affine variety over $k$.  Then $X$ has an arithmetically-pure model. 
  \end{lemma}
  \begin{proof}
 Affine varieties have an arithmetically-pure model by Hartog's Lemma.  
  \end{proof}
  
  \begin{lemma}\label{lem:semiab}
Let $X$ be a variety over $k$ which admits   a  finite morphism to some semi-abelian variety over $k$. Then $X$ has an arithmetically-pure model. 
  \end{lemma}
  \begin{proof} Let $G$ be a semi-abelian variety and let $X\to G$ be a finite morphism. It follows from Hartog's Lemma that $X$ has an arithmetically-pure model if and only if $G$ has an arithmetically-pure model.  Choose a $\ZZ$-finitely generated subring and a  model $\mathcal{G}$  for $G$ over $A$ such that $\mathcal{G}\to \Spec A$ is a semi-abelian scheme. Then, this model $\mathcal{G}$  has the desired extension property by   \cite[Lemma~A.2]{Mochizuki1}, so that $G$ (hence $X$) has an arithmetically-pure model.
  \end{proof}

  \begin{remark}\label{remark:vB}
  Let $X$ be a projective integral groupless surface over $k$ which admits a non-constant map to some abelian variety. Then $X$ has an arithmetically-pure model by \cite[Lemma~8.11]{vBJK}.
  \end{remark}
  
    \begin{corollary}\label{cor:psar_is_mor}
  Let X be an integral variety over $k$, and let $\Delta\subset X$ be a closed subset. Assume that one of the following statements holds.
  \begin{enumerate}
  \item The variety $X$ is affine over $k$.
  \item There is  a finite morphism $X\to G$ with $G$ a semi-abelian variety over $k$.
  \item We have that $X$ is a groupless surface which  admits a non-constant morphism $X\to A$ with $A$ an abelian variety over $k$.
    \end{enumerate}
    Then $X$ is arithmetically hyperbolic modulo $\Delta$ over $k$ if and only if $X$ is Mordellic modulo $\Delta$ over $k$.
  \end{corollary}
  \begin{proof}
  Assume $(1)$. Then the statement follows from Lemma \ref{lem:affine} and Theorem \ref{thm:ar_hyp_is_mor}.  Similarly, if $(2)$ holds, then the statement follows  from Lemma \ref{lem:semiab} and Theorem \ref{thm:ar_hyp_is_mor}.  
 Finally, assuming $(3)$, the statement follows from Remark \ref{remark:vB} and Theorem \ref{thm:ar_hyp_is_mor}. 
  \end{proof}

  \begin{remark} \label{remark:md}    Let $g\geq 1$ and $N\geq 3$ be integers.
  Now, if $X$ is the fine moduli space of $g$-dimensional principally polarized abelian schemes over $k$ with level $\Qbar$ structure, then $X$ has an arithmetically-pure model. As is explained in \cite{Martin}, this is a consequence of Grothendieck's theorem on homomorphisms of abelian schemes \cite{GrothendieckHom}. The existence of such a model is used by Martin-Deschamps to deduce the Mordellicity of $X_k$ over $k$ from the Mordellicity of $X$ over $\Qbar$  (cf. Theorem \ref{thm:shaf_Faltings}).
  \end{remark}

   \section{Pseudo-Brody   hyperbolicity}\label{section:ps_brody}
   The notion of pseudo-hyperbolicity appeared first in the work of Kiernan and Kobayashi \cite{KiernanKobayashi} and afterwards in  Lang \cite{Lang2}; see also \cite{KobayashiBook}. We recall some of the definitions.
   
   \begin{definition}
Let $X$ be a variety over $\CC$ and let $\Delta$ be a closed subset of $X$. We say that $X$ is \emph{Brody hyperbolic modulo $\Delta$} if every holomorphic non-constant map $\CC\to X^{\an}$  factors over $\Delta$. 
\end{definition}

\begin{definition} A variety $X$ over $\CC$ is   \emph{pseudo-Brody hyperbolic} if there is a proper closed subset $\Delta\subsetneq X$ such that $X$ is Brody hyperbolic modulo $\Delta$.
\end{definition}

Green-Griffiths and Lang conjectured that a projective variety of general type is pseudo-Brody hyperbolic.   The conjecture that a projective variety is of general type if and only if it is pseudo-Brody hyperbolic is commonly referred to as the Green-Griffiths-Lang conjecture.

Note that the notion of pseudo-Brody hyperbolicity is a birational invariant. More precisely, 
if $X$ and $Y$ are proper integral varieties over $\CC$ which are birational, then $X$ is pseudo-Brody hyperbolic if and only if $Y$ is pseudo-Brody hyperbolic. Furthermore, just like the notions of pseudo-Mordellicity and pseudo-grouplessness, the notion of pseudo-Brody hyperbolicity descends along finite \'etale morphisms. That is, if $X\to Y$ is finite \'etale, then $X$ is pseudo-Brody hyperbolic if and only if $Y$ is pseudo-Brody hyperbolic. Also, it is not hard to show that, if a variety $X$ is Brody hyperbolic modulo $\Delta$, then $X$ is groupless modulo $\Delta$.

Note that a variety $X$ is Brody hyperbolic (as defined in Section \ref{section:brody}) if and only if $X$ is Brody hyperbolic modulo the empty set. Bloch--Ochiai--Kawamata's theorem classifies Brody hyperbolic closed subvarieties of abelian varieties. In fact, their result  is a consequence of the following more general statement (also proven in \cite{Kawamata}).

   \begin{theorem}[Bloch--Ochiai--Kawamata]
   Let $X$ be a closed subvariety of an abelian variety $A$. Let $\mathrm{Sp}(X)$ be the special locus of $X$. Then $\mathrm{Sp}(X)$ is a closed subset of $X$ and $X$ is Brody hyperbolic modulo $\mathrm{Sp}(X)$.  
   \end{theorem}
   
   We now introduce the pseudo-analogue of Kobayashi hyperbolicity for algebraic varieties. Of course, these definitions make sense for complex-analytic spaces.
   
      \begin{definition}
Let $X$ be a variety over $\CC$ and let $\Delta$ be a closed subset of $X$.  We say that $X$ is \emph{Kobayashi hyperbolic modulo $\Delta$} if, for every $x$ and $y$ in $X^{\an}\setminus \Delta^{\an}$ with $x\neq y$, the Kobayashi pseudo-distance $d_{X^{\an}}(p,q)$ is positive.   \end{definition} 

\begin{definition} A variety $X$ over $\CC$ is   \emph{pseudo-Kobayashi hyperbolic} if there is a proper closed subset $\Delta\subsetneq X$ such that $X$ is Kobayashi hyperbolic modulo $\Delta$.
\end{definition}

It is clear from the definitions and the fact that the Kobayashi pseudo-metric vanishes everywhere on $\CC$, that a variety $X$ which is Kobayashi hyperbolic modulo a closed subset $\Delta\subset X$ is Brody hyperbolic modulo $\Delta$. 
Nonetheless, the notion of pseudo-Kobayashi hyperbolicity remains quite mysterious at the moment. Indeed,  we do not know whether a pseudo-Brody hyperbolic projective variety $X$ over $\CC$ is pseudo-Kobayashi hyperbolic. 

One can show that the notion of pseudo-Kobayashi hyperbolicity is a birational invariant. That is,  
if $X$ and $Y$ are proper integral varieties over $\CC$ which are birational, then $X$ is pseudo-Kobayashi hyperbolic if and only if $Y$ is pseudo-Kobayashi hyperbolic; see \cite{KobayashiBook}. Moreover, just like the notions of pseudo-Mordellicity and pseudo-grouplessness, pseudo-Kobayashi hyperbolicity descends along finite \'etale morphisms.

Yamanoi proved the pseudo-Kobayashi analogue of Bloch--Ochiai--Kawamata's theorem for closed subvarieties of abelian varieties; see \cite[Theorem~1.2]{Yamanoi2}.

   \begin{theorem}[Yamanoi]\label{thm:yamanoi_Kob}
      Let $X$ be a closed subvariety of an abelian variety $A$. Let $\mathrm{Sp}(X)$ be the special locus of $X$. Then $\mathrm{Sp}(X)$ is a closed subset of $X$ and $X$ is Kobayashi hyperbolic modulo $\mathrm{Sp}(X)$.
   \end{theorem}
  
The Lang-Vojta conjecture and the Green-Griffiths conjecture predict that the above notions of hyperbolicity are   equivalent. To state this conjecture, we will  need one more definition. (Recall that $k$ denotes an algebraically closed field of characteristic zero.)
\begin{definition} A variety $X$ over $k$ is \emph{strongly-pseudo-Brody hyperbolic} (resp. \emph{strongly-pseudo-Kobayashi hyperbolic}) if, for every subfield $k_0\subset k$, every model $\mathcal{X}$ for $X$ over $k_0$, and every embedding $k_0\to \CC$, the variety $\mathcal{X}_{0,\CC}$ is pseudo-Brody hyperbolic (resp. pseudo-Kobayashi hyperbolic).
\end{definition}

  \begin{conjecture}[Strong Lang-Vojta, III]\label{conj:S3}
  Let $X$ be a projective variety over $k$. Then the following statements are equivalent.
  \begin{enumerate}
  \item The variety $X$ is strongly-pseudo-Brody hyperbolic over $k$.
  \item The variety $X$ is strongly-pseudo-Kobayashi hyperbolic over $k$.
  \item The projective variety $X$ is pseudo-Mordellic over $k$.
    \item The projective variety $X$ is pseudo-arithmetically hyperbolic over $k$.
  \item  The projective variety $X$ is pseudo-groupless over $k$.
  \item The projective variety $X$ is of general type over $k$.
  \end{enumerate}
  
  \end{conjecture}
 
  As stated this conjecture does not predict that every conjugate of a pseudo-Brody hyperbolic variety is again pseudo-Brody hyperbolic. We state this as a separate conjecture, as we did in Conjecture \ref{conj:conjugates} for Brody hyperbolic varieties.
  
  \begin{conjecture}[Conjugates of pseudo-Brody hyperbolic varieties]\label{conj:conjugates2}  
  If $X$ is a variety over $k$ and $\sigma$ is a field automorphism of $k$, then the following statements hold.
  \begin{enumerate}
    \item The variety $X$ is pseudo-Brody hyperbolic if and only if $X^{\sigma}$ is pseudo-Brody hyperbolic.
  \item The variety $X$ is pseudo-Kobayashi hyperbolic if and only if $X^{\sigma}$ is pseudo-Kobayashi hyperbolic.
  \end{enumerate}
  \end{conjecture}
  
  We conclude this section with a brief discussion of a theorem of Kwack on the algebraicity of holomorphic maps to a hyperbolic variety, and a possible extension of his result to the pseudo-setting.
  
  \begin{remark}[Borel hyperbolicity] Let $X$ be a variety over $\mathbb{C}$ and let $\Delta\subset X$ be a closed subset. 
  We extend the notion of Borel hyperbolicity introduced in \cite{JKuch} to the pseudo-setting and say that  $X$  is \emph{Borel hyperbolic modulo $\Delta$} if, for every reduced variety $S$ over $\mathbb{C}$, every holomorphic map $f:S^{\an}\to X^{\an}$ with $f(S^{\an})\not\subset \Delta^{\an}$ is the analytification of a morphism $\varphi:S\to X$.  The proof of \cite[Lemma~3.2]{JKuch} shows that, if $X$ is Borel hyperbolic modulo $\Delta$, then it is Brody hyperbolic modulo $\Delta$. In \cite{Kwack} Kwack showed that, if $X$ is a proper Kobayashi hyperbolic variety, then $X$ is Borel hyperbolic (modulo the empty set). It seems reasonable to suspect that Kwack's theorem also holds in the pseudo-setting. 
 Thus, we may ask: if $X$ is Kobayashi hyperbolic modulo $\Delta$, does it follow that $X$ is Borel hyperbolic modulo $\Delta$?  \end{remark}
  
The reader interested in   investigating  further  complex-analytic  notions of hyperbolicity is also encouraged to have a look at the notion of taut-hyperbolicity modulo a subset introduced  by Kiernan-Kobayashi \cite{KiernanKobayashi}; see also \cite[Chapter~5]{KobayashiBook}.

\section{Algebraic hyperbolicity}\label{section:alg_hyp}
In the following three sections we investigate (a priori) different function field analogues of Mordellicity. Conjecturally, they are all equivalent notions.  At this point it is also clear that hyperbolicity \emph{modulo} a subset is more natural to study (especially from a birational perspective) which is why we will give the definitions in this more general context.
 
 The   notion we introduce in this section extends Demailly's notion of algebraic hyperbolicity \cite{Demailly, JKa} to the pseudo-setting. 

 \begin{definition}[Algebraic hyperbolicity modulo a subset] \label{def:alg_hyp}
 Let $X$ be a projective scheme over $k$ and let $\Delta$ be a closed subset of $X$. We say that $X$ is \emph{algebraically hyperbolic over $k$ modulo $\Delta$} if, for every ample line bundle $L$ on $X$, there is a real number $\alpha_{X,\Delta, L}$ depending only on $X$, $\Delta$, and $L$ such that, for every smooth projective connected curve $C$ over $k$ and every morphism $f:C\to X$ with $f(C)\not \subset \Delta$, the inequality
 \[
 \deg_C f^\ast L \leq \alpha_{X,\Delta, L} \cdot \mathrm{genus}(C)
 \] holds.
 \end{definition}

\begin{definition}
 A projective scheme  $X$ is \emph{pseudo-algebraically hyperbolic (over $k$)} if there is a proper closed subset $\Delta$ such that $X$ is algebraically hyperbolic modulo $\Delta$.
\end{definition}

We will say that a projective scheme $X$ is \emph{algebraically hyperbolic} over $k$ if it is algebraically modulo the empty subset.   This terminology is consistent with  that of  \cite{JKa}.

The motivation for introducing and studying algebraically hyperbolic projective schemes are the   results of Demailly stated below. They say that algebraic hyperbolicity lies between Brody hyperbolicity and grouplessness. In particular, the Lang-Vojta conjectures as stated in the previous sections imply that groupless projective  varieties should be algebraically hyperbolic, and that algebraically hyperbolic projective varieties should be Brody hyperbolic. This observation is due to Demailly and allows one to split the conjecture that groupless projective varieties are Brody hyperbolic into two a priori different parts. 

Before stating Demailly's theorems, we note that it is not hard to see that pseudo-algebraic hyperbolicity descends along finite \'etale maps, and that pseudo-algebraic hyperbolicity  for projective schemes is a birational invariant; see \cite[\S4]{JXie} for details.   These two properties should be compared with their counterparts for   pseudo-grouplessness, pseudo-Mordellicity, pseudo-Brody hyperbolicity, and pseudo-Kobayashi hyperbolicity.

Demailly's theorem for projective schemes reads as follows.

\begin{theorem}[Demailly]
Let $X$ be a projective scheme over $\CC$. If $X$ is Brody hyperbolic, then $X$ is algebraically hyperbolic over $\CC$.
\end{theorem}

A proof of this is given in \cite[Theorem~2.1]{Demailly} when $X$ is smooth. The smoothness of $X$ is however not used in its proof. 
We  stress that it is \emph{not} known whether a pseudo-Brody hyperbolic projective scheme is pseudo-algebraically hyperbolic.  On the other hand, Demailly  proved that algebraically hyperbolic projective schemes are groupless, and his proof can be adapted to show the following more general statement.

\begin{theorem}[Demailly + $\epsilon$]\label{thm:demailly1}
Let $X$ be a projective scheme over $k$ and let $\Delta\subset X$ be a closed subset. If $X$ is algebraically hyperbolic modulo $\Delta$, then $X$ is groupless modulo $\Delta$.
\end{theorem}
\begin{proof}
See  \cite{JKa} when $\Delta =\emptyset$. The more general statement is proven in \cite{JXie}. The argument involves   the multiplication  maps on an abelian variety.
\end{proof}

Combining Demailly's theorems with Bloch--Ochiai--Kawamata's theorem, we obtain that a closed subvariety of an abelian variety over $k$ is algebraically hyperbolic over $k$ if and only if it is groupless.  The pseudo-version of this theorem is due to Yamanoi (see Section \ref{section:questions} for a precise statement).

\section{Boundedness}\label{section:boundedness}
To say that a projective variety $X$  is algebraically hyperbolic (Definition \ref{def:alg_hyp}) is to say that the degree of any curve $C$ is bounded uniformly and linearly in the genus of that curve. The reader interested in understanding how far we are from proving that groupless projective schemes  are algebraically hyperbolic is naturally led to studying variants of algebraic hyperbolicity in which one asks (in Definition \ref{def:alg_hyp} above) for ``weaker'' bounds on the degree of a map.   This led the authors of \cite{JKa}  to introducing the   notion of boundedness. To state their definition, we first recall some basic properties of moduli spaces of morphisms between projective schemes.

Let $S$ be  a scheme, and let   $X\to S$   and $Y\to S$ be projective flat morphisms of schemes.  By Grothendieck's theory of Hilbert schemes and Quot schemes \cite{Nitsure}, the functor 
$$
\mathrm{Sch}/S^{op} \to \mathrm{Sets}, \quad T\to S \mapsto \Hom_T(Y_T, X_T)
$$ is representable by  an $S$-scheme which we denote by $\underline{\Hom}_S(X,Y)$. Moreover, for $h\in \QQ[t]$ a polynomial, the subfunctor parametrizing morphisms whose graph has  Hilbert polynomial $h$ is representable by a quasi-projective subscheme  $\underline{\Hom}_S^{h}(Y,X)$ of $\underline{\Hom}_S(Y,X)$. Similarly, the subfunctor  of $\underline{\Hom}_S(X,X)$ parametrizing automorphisms of $X$ over $S$ is representable by a locally finite type group scheme scheme $ {\Aut}_{X/S}$ over $S$. It is imperative to note that this group scheme need not be quasi-compact. In fact, for a K3 surface $X$ over $\mathbb{C}$, the scheme $ {\Aut}_{X/\CC}$ is zero-dimensional. Nonetheless, there are K3 surfaces with infinitely many automorphisms. Thus, the automorphism group scheme of a projective scheme over $k$  is not necessarily of finite type (even when  it is zero-dimensional).

 If $S=\Spec k$, $d\geq 1$ is an integer, and $X=Y = \mathbb{P}^1_k$, let $\underline{\Hom}_k^d(\mathbb{P}^1_k,\mathbb{P}^1_k)$ be the subscheme of $\underline{\Hom}_k(\mathbb{P}^1_k,\mathbb{P}^1_k)$ parametrizing morphisms of degree $d$. In particular, we have that $\underline{\mathrm{Hom}}_k^1(\mathbb{P}^1_k,\mathbb{P}^1_k) =  {\Aut}_{\mathbb{P}^1_k/k}= \mathrm{PGL}_{2,k}$. For every $d\geq 1$, the quasi-projective scheme $\underline{\Hom}_k^d(\mathbb{P}^1_k,\mathbb{P}^1_k)$ is non-empty (and even positive-dimensional). If we identity the subscheme of $\underline{\Hom}_k(\mathbb{P}^1_k,\mathbb{P}^1_k)$ parametrizing constant morphisms with $\mathbb{P}^1_k$, then  
\[ 
\underline{\Hom}_k(\mathbb{P}^1,\mathbb{P}^1_k) = \mathbb{P}^1_k \sqcup \mathrm{PGL}_{2,k} \sqcup \bigsqcup_{d=2}^{\infty} \underline{\Hom}^d_k(\mathbb{P}^1_k,\mathbb{P}^1_k).
\] It follows that the scheme $\underline{\Hom}_k(\mathbb{P}^1,\mathbb{P}^1_k) $ has infinitely many connected components. It is in particular not of finite type.

 It turns out that studying projective varieties $X$ over $k$ for which \emph{every} Hom-scheme $\mathrm{\Hom}_k(Y,X)$ is of finite type is closely related to studying algebraically hyperbolic varieties. The aim of this section is to explain the connection in a systematic manner as is done in \cite{vBJK, JKa, JXie}. We start with the following definitions.

 \begin{definition}[Boundedness modulo a subset] Let $n\geq 1$ be an integer, 
 let $X$ be a projective scheme over $k$, and let $\Delta$ be a closed subset of $X$. We say that $X$ is \emph{$n$-bounded   over $k$ modulo $\Delta$} if, for every normal projective variety $Y$ of dimension at most $n$, the scheme $\underline{\Hom}_k(Y,X)\setminus \underline{\Hom}_k(Y,\Delta)$ is of finite type over $k$. We say that $X$ is \emph{bounded over $k$ modulo $\Delta$} if, for every $n\geq 1$, the scheme $X$ is $n$-bounded modulo $\Delta$.
 \end{definition}

\begin{definition} Let $n\geq 1$ be an integer.
 A projective scheme  $X$ over $k$ is \emph{pseudo-$n$-bounded over $k$} if there is a proper closed subset $\Delta$ such that $X$ is $n$-bounded  modulo $\Delta$.   
\end{definition}

 \begin{definition}
 A projective scheme $X$ over $k$ is \emph{pseudo-bounded over $k$} if it is pseudo-$n$-bounded over $k$ for every $n\geq 1$.
 \end{definition}
 
 \begin{remark}
 At the beginning of this section we discussed the structure of the scheme $\underline{\Hom}_k(\mathbb{P}^1_k,\mathbb{P}^1_k)$. From that discussion it follows that $\mathbb{P}^1_k$ is not $1$-bounded over $k$. In particular, if $X$ is a  $1$-bounded  projective variety over $k$, then it has no rational curves.  It is also not hard to show that $\mathbb{P}^1_k$ is not pseudo-$1$-bounded by showing that, for every $x$ in $\mathbb{P}^1(k)$, there is a $y$ in $\mathbb{P}^1(k)$ such that  
the set of morphisms $f:\mathbb{P}^1_k\to \mathbb{P}^1_k$ with $f(y) = x$ is infinite. We refer the interested reader to Section \ref{section:geom_hyp} for a related discussion. 
 \end{remark}

We   say that $X$ is \emph{bounded} if it is bounded modulo the empty subset. We employ similar terminology for $n$-bounded.  This terminology is consistent with that of \cite{vBJK, JKa}.
Let us start with looking at some implications and relations between these a priori different notions of boundedness.

 Boundedness is a condition on moduli spaces of maps from higher-dimensional varieties. Although it might seem a priori  stronger  than $1$-boundedness,  Lang-Vojta's conjecture predicts their equivalence. In fact, we have the following   result from \cite{JKa} which shows the equivalence of three a priori different notions. In this theorem, the implications $(2)\implies (1)$ and $(3)\implies (1)$ are straightforward consequences of the definitions. 

\begin{theorem}\label{thm:eq_bounded} Let $X$ be a projective scheme over $k$. Then the following are equivalent.
\begin{enumerate}
\item  The projective scheme $X$ is $1$-bounded over $k$.
\item  The projective scheme $X$ is bounded over $k$.
\item  For every ample line bundle $\mathcal{L}$ and every integer $g\geq 0$, there is an integer $\alpha(X,\mathcal{L},g)$ such that, for every smooth projective connected curve $C$ of genus $g$ over $k$ and every morphism $f:C\to X$, the inequality
 \[
 \deg_C f^\ast \mathcal{L} \leq \alpha(X,\mathcal{L},g)
 \] holds.
\end{enumerate}
\end{theorem} 
\begin{proof}
The fact that a $1$-bounded scheme is  $n$-bounded  for every $n\geq 1$ is proven by induction on $n$ in \cite[\S 9]{JKa}.  The idea is that, if  $f_i:Y\to X$ is a sequence of morphisms from an $n$-dimensional smooth projective variety $Y$ with pairwise distinct Hilbert polynomial, then one can find a smooth hyperplane section $H\subset Y$ such that  the restrictions  $f_i|_{H}$ of these morphisms $f_i$  to $H$ still have pairwise distinct Hilbert polynomial. 

The fact that a bounded scheme satisfies the ``uniform'' boundedness property in $(3)$ follows from reformulating this statement in terms of the quasi-compactness of the universal Hom-stack of morphisms of curves of genus $g$ to $X$; see  the proof of \cite[Theorem~1.14]{JKa} for details.
\end{proof}

Studying boundedness is ``easier'' than studying boundedness modulo a subset $\Delta$. Indeed, 
part of the analogue of this theorem for pseudo-boundedness (unfortunately) requires an assumption on the base field $k$. 

\begin{theorem} \label{thm:1_is_b} Let $X$ be a projective scheme over $k$, and let $\Delta$ be a closed subset of $X$.  {Assume} that $k$ is \textbf{uncountable}. Then $X$ is $1$-bounded modulo $\Delta$ if and only if $X$ is bounded modulo $\Delta$.
\end{theorem}
\begin{proof} This is proven in \cite{vBJK}, and the argument is similar to the proof of Theorem \ref{thm:eq_bounded}. We briefly indicate how the uncountability of $k$ is used.  

Assume that $X$ is $1$-bounded modulo $\Delta$. We show by induction on $n$ that $X$ is $n$-bounded modulo $\Delta$ over $k$. If $n=1$, then this holds by assumption. Thus, let $n>1$ be an integer and assume that $X$ is $(n-1)$-bounded modulo $\Delta$.
Let $Y$ be an $n$-dimensional projective reduced scheme and let $f_m:Y \to X$ be a sequence of morphisms with pairwise distinct Hilbert polynomial   such that, for every $m=1,2,\ldots$, we have   $f_m(Y)\not\subset \Delta$. Since $k$ is \textbf{uncountable}, there is an ample divisor $D$ in $Y$ which  is not contained in $f_m^{-1}(\Delta)$ for   all $m\in \{1,2,\ldots\}$. Now, the restrictions $f_m|_D:D\to X$ have pairwise distinct Hilbert polynomial and, for infinitely many $m$, we have that $f_m(D)\not\subset \Delta$.   This contradicts the induction hypothesis. We conclude that $X$ is bounded modulo $\Delta$ over $k$, as required.
\end{proof}
The ``pseudo'' analogue of the equivalence between $(2)$ and $(3)$ in Theorem \ref{thm:eq_bounded} holds without any additional assumption on $k$; see \cite{vBJK}.
\begin{theorem}\label{thm:eq_bounded2} Let $X$ be a projective scheme over $k$. Then  
  $X$ is bounded modulo $\Delta$ over $k$
if and only if,   for every ample line bundle $\mathcal{L}$ and every integer $g\geq 0$, there is an integer $\alpha(X,\mathcal{L},g)$ such that, for every smooth projective connected curve $C$ of genus $g$ over $k$ and every morphism $f:C\to X$ with $f(C)\not \subset \Delta$, the inequality
 \[
 \deg_C f^\ast \mathcal{L} \leq \alpha(X,\mathcal{L},g)
 \] holds.
\end{theorem}

It is not hard to see that being pseudo-$n$-bounded descends along finite \'etale maps. Also, if $X$ and $Y$ are projective schemes over $k$ which are birational, then $X$ is pseudo-$1$-bounded if and only if $Y$ is pseudo-$1$-bounded; see \cite[\S4]{JXie}. However, in general, it is not clear that pseudo-$n$-boundedness is a birational invariant  (unless $n=1$ or $k$ is uncountable).

 It is shown in \cite{vBJK, JKa} that pseudo-algebraically hyperbolic varieties are pseudo-bounded. More precisely, one can prove the following statement.
 \begin{theorem}
   If $X$ is algebraically hyperbolic  modulo $\Delta$ over $k$, then  $X$ is  bounded modulo $\Delta$.
 \end{theorem}
 \begin{proof}
 This is proven in three steps in \cite[\S 9]{vBJK}. First, one  chooses an uncountable algebraically closed field $L$ containing $k$ and shows that $X_L$ is algebraically hyperbolic modulo $\Delta_L$. Then, one makes the ``obvious'' observation that $X_L$ is $1$-bounded modulo $\Delta_L$. Finally, as $L$ is uncountable and $X_L$ is $1$-bounded modulo $\Delta_L$, it follows from Theorem \ref{thm:1_is_b} that    $X_L$ is bounded modulo $\Delta_L$. 
 \end{proof}

   Demailly proved that algebraically hyperbolic projective varieties are groupless (Theorem \ref{thm:demailly1}). His proof can be adapted to  show the following more general statement.

  \begin{proposition}[Demailly + $\epsilon$]
  If $X$ is $1$-bounded modulo $\Delta$ over $k$, then  $X$ is  groupless modulo $\Delta$.
\end{proposition}

\section{Geometric hyperbolicity}\label{section:geom_hyp}
In the definition of Mordellicity over $\Qbar$ one considers the ``finiteness of arithmetic curves'' on some model. On the other hand,  the notions of algebraic hyperbolicity and boundedness require one to test ``boundedness of   curves''. In this section we introduce a new notion in which one considers the   ``finiteness of pointed curves''.

 \begin{definition}[Geometric hyperbolicity modulo a subset]   Let $X$ be a variety over $k$ and let $\Delta$ be a closed subset of $X$. We say that $X$ is \emph{geometrically hyperbolic   over $k$ modulo $\Delta$} if, for every $x$ in $X(k)\setminus \Delta$, every smooth connected curve $C$  over $k$ and every $c$ in $C(k)$, we have that the set $\Hom_k((C,c),(X,x))$ of morphisms $f:C\to X$ with $f(c)=x$ is finite.  
 \end{definition}

\begin{definition}  
 A variety  $X$ over $k$ is \emph{pseudo-geometrically hyperbolic over $k$} if there is a proper closed subset $\Delta$ such that $X$ is geometrically hyperbolic   modulo $\Delta$. 
\end{definition}

 We say that a variety $X$ over $k$  is \emph{geometrically hyperbolic over $k$} if it is geometrically hyperbolic modulo the empty subset.  At this point we should note that a projective scheme $X$ over $k$ is geometrically hyperbolic over $k$ if and only if it is ``$(1,1)$-bounded''. The latter notion is defined in \cite[\S 4]{JKa}, and the equivalence of these two notions is \cite[Lemma~4.6]{JKa} (see also Proposition \ref{prop:1b_is_geomhyp} below). The terminology ``(1,1)-bounded  modulo $\Delta$'' is   used in \cite{vBJK}, and also coincides with being geometrically hyperbolic modulo $\Delta$ for projective schemes by the results in \cite[\S 9]{vBJK}.

 \begin{remark}[Geometric hyperbolicity versus arithmetic hyperbolicity]\label{remark:similarity} Let us say that a scheme $T$ is an \emph{arithmetic curve} if there is a number field $K$ and a finite set of finite places $S$ of $K$ such that $T = \Spec \OO_{K,S}$.  
  Let $X$ be a variety over $\Qbar$. It is not hard to show that the following two statements are equivalent.
  \begin{enumerate}
  \item The variety $X$ is arithmetically hyperbolic (or Mordellic) over $\Qbar$.
  \item For every arithmetic curve $\mathcal{C}$, every closed point $c$ in $\mathcal{C}$,   every model $\mathcal{X}$ for $X$ over $\mathcal{C}$, and every closed point $x$ of $\mathcal{X}$, the subset 
  $$
  \Hom_{\mathcal{C}}((\mathcal{C},c), (\mathcal{X},x)) \subset \mathcal{X}(\mathcal{C})
  $$ of morphisms $f:\mathcal{C}\to \mathcal{X}$ with $f(c)=x$ is finite.
  \end{enumerate}
  Indeed, if $(1)$ holds, then $\Hom_{\mathcal{C}}(\mathcal{C},\mathcal{X})$ is finite by definition, so that clearly the set  $$
  \Hom_{\mathcal{C}}((\mathcal{C},c), (\mathcal{X},x))
  $$  is finite. Conversely, assume that $(2)$ holds. Now, let $\mathcal{C}$ be an arithmetic curve and let $\mathcal{X}$ be a model for $X$ over $\mathcal{C}$. To show that $\mathcal{X}(\mathcal{C}) $ is finite, let $c$ be a closed point of $\mathcal{C}$ and let $\kappa$ be its  residue field. Then $\kappa$ is finite and   $c$ lies in $\mathcal{C}(\kappa)$. In particular, the image of $c$ along any morphism $\mathcal{C}\to \mathcal{X}$ is a $\kappa$-point of $\mathcal{X}$. This shows that 
  \[\mathcal{X}(\mathcal{C}) \subset \bigcup_{x \in \mathcal{X}(\kappa)} \Hom_{\mathcal{C}} ((\mathcal{C},c),(\mathcal{X},x)).\] Since $\mathcal{X}(\kappa)$ is finite and every set $\Hom_{\mathcal{C}} ((\mathcal{C},c),(\mathcal{X},x))$ is finite, we conclude that $\mathcal{X}(\mathcal{C})$ is finite, as required.
  
  The second statement allows one to see the similarity between geometric hyperbolicity and arithmetic hyperbolicity. Indeed, the variety $X$ is geometrically hyperbolic over $\Qbar$ if, for every integral algebraic curve $C$ over $\Qbar$, every closed point $c$ in $C$, and every closed point $x$ of $X$, the set $$\Hom_k((C,c),(X,x)) = \Hom_C((C,c),(X\times C, (x,c))) $$ is finite.  
 \end{remark}
 
 Just like pseudo-grouplessness and pseudo-Mordellicity, it is not hard to see that pseudo-geometric hyperbolicity descends along finite \'etale morphisms. Also, if $X$ and $Y$ are projective varieties  which are birational, then $X$ is pseudo-geometrically hyperbolic if and only if $Y$ is pseudo-geometrically hyperbolic.    
 
 The following proposition says that a projective scheme is geometrically hyperbolic if and only if the moduli space of pointed maps is of finite type. In other words, asking for boundedness of all pointed maps is equivalent to asking for the finiteness of all sets of pointed maps.
 
    \begin{proposition}\label{prop:1b_is_geomhyp} Let $X$ be a projective scheme over $k$ and let $\Delta$ be a   closed subset of $X$. Then the following are equivalent.
   \begin{enumerate}
   \item For every smooth projective connected curve $C$ over $k$, every $c$ in $C(k)$ and every $x$ in $X(k)\setminus \Delta$, the scheme $\underline{\Hom}_k((C,c),(X,x))$ is of finite type over $k$.
   \item The variety $X$ is geometrically hyperbolic modulo $\Delta$.
   \end{enumerate}
  \end{proposition}
  \begin{proof} This is proven in \cite[\S9]{vBJK}. The proof is a standard application of the bend-and-break principle. Indeed, the implication $(2)\implies (1)$ being obvious, let us show that $(1)\implies (2)$. Thus, let us assume that $X$ is not geometrically hyperbolic modulo $\Delta$, so that there is a sequence $f_1, f_2, \ldots$ of pairwise distinct elements of $\Hom_k((C,c),(X,x))$, where $C$ is a smooth projective connected curve over $k$, $c\in C(k)$ and $x\in X(k)\setminus \Delta$.   Since $\underline{\Hom}_k((C,c),(X,x)) $ is of finite type,  the degree of all the $f_i$ is bounded by some real number (depending only on $X, \Delta, c, x$ and $C$). In particular, it follows that some connected component of $\underline{\Hom}_k((C,c),(X,x))$ has infinitely many elements. As each connected component of $\underline{\Hom}_k((C,c),(X,x))$ is a finite type scheme over $k$, it follows from  bend-and-break   \cite[Proposition~3.5]{Debarrebook1} that there is a rational curve in $X$ containing $x$. This contradicts the  fact that every rational curve in $X$ is contained in $\Delta$ (by Proposition \ref{prop:psgeomhyp_is_psgr1}).
  \end{proof}
  
  This proposition has the following consequence.

 \begin{corollary}\label{cor:urata}
 Let $X$ be a projective scheme over $k$ and let $\Delta$ be a proper closed subset of $X$.  If $X$ is $1$-bounded modulo $\Delta$, then $X$ is geometrically hyperbolic modulo $\Delta$.
 \end{corollary}
 \begin{proof}
 If $X$ is $1$-bounded, then it is clear that, for every smooth projective connected curve $C$, every $c$ in $C(k)$ and every $x$ in $X(k)\setminus \Delta$, the scheme $\underline{\Hom}_k((C,c),(X,x))$ is of finite type over $k$. Indeed, the latter scheme is closed in  the scheme $\underline{\Hom}_k(C,X)$, and contained in the quasi-projective subscheme $\underline{\Hom}_k(C,X)\setminus \underline{\Hom}_k(C,\Delta)$. Therefore, the result follows from Proposition \ref{prop:1b_is_geomhyp}.
 \end{proof}
 
\begin{remark}
Urata showed that a Brody hyperbolic projective variety over $\CC$ is geometrically hyperbolic over $\CC$; see \cite[Theorem~5.3.10]{KobayashiBook} (or the original \cite{Urata}). Note that Corollary  \ref{cor:urata} generalizes Urata's theorem (in the sense that the assumption in Corollary \ref{cor:urata} is a priori weaker than being Brody hyperbolic, and we also allow for an ``exceptional set'' $\Delta$). Indeed, as a Brody hyperbolic projective variety  is $1$-bounded (even algebraically hyperbolic), Urata's theorem follows directly from Corollary \ref{cor:urata}.  
\end{remark}

 Demailly's argument to show that algebraically hyperbolic projective varieties are groupless (Theorem  \ref{thm:demailly1}) can be adapted to show that geometrically hyperbolic projective varieties  are groupless; see \cite{JXie} for a detailed proof.

 \begin{proposition}\label{prop:psgeomhyp_is_psgr1}   Let $X$ be a projective  variety over $k$ and let $\Delta$ be a   closed subset of $X$. 
 If $X$ is   geometrically hyperbolic modulo $\Delta$ over $k$, then $X$ is  groupless  modulo $\Delta$ over $k$.  
 \end{proposition}

\section{The conjectures summarized}   \label{section:conjectures}
 
After a lengthy preparation, we are finally ready to state the complete version of Lang-Vojta's conjecture.

  \begin{conjecture}[Strong Lang-Vojta, IV]\label{conj:S6}
  Let $X$ be a projective variety over $k$. Then the following statements are equivalent.
  \begin{enumerate}
  \item The variety $X$ is strongly-pseudo-Brody hyperbolic over $k$.
  \item  The variety $X$ is strongly-pseudo-Kobayashi hyperbolic.
  \item The projective variety $X$ is pseudo-Mordellic over $k$.
    \item The projective variety $X$ is pseudo-arithmetically hyperbolic over $k$.
  \item  The projective variety $X$ is pseudo-groupless over $k$.
  \item The projective variety $X$ is pseudo-algebraically hyperbolic over $k$.
  \item The projective variety $X$ is pseudo-bounded over $k$.
\item The projective variety $X$ is pseudo-$1$-bounded over $k$.
\item The projective variety $X$ is pseudo-geometrically hyperbolic over $k$.
  \item The projective variety $X$ is of general type over $k$.
  \end{enumerate}
\end{conjecture}

Conjecture \ref{conj:S6} is the final version of the  Lang-Vojta conjecture for pseudo-hyperbolic varieties, and also encompasses Green-Griffiths's conjecture for  projective varieties of general type. We note that one aspect of the Lang-Vojta conjecture and the Green-Griffiths conjecture that is ignored in this conjecture is whether the conjugate of a Brody hyperbolic variety is Brody hyperbolic (see Conjectures \ref{conj:conjugates} and \ref{conj:conjugates2}).

 The following implications are known. 
First,  $(6) \implies (7)$, $(7)\implies (8)$, $(8)\implies (9)$, and $(9)\implies (5)$. 
 Also, 
 $(3)\implies (4)$, $(3)\implies (5)$.
 Finally, $(2)\implies (1)$ and $(1)\implies (5)$.  
The following diagram summarizes these known implications. The content of the Strong Lang-Vojta conjecture is that all the notions appearing in this diagram are equivalent.  \\

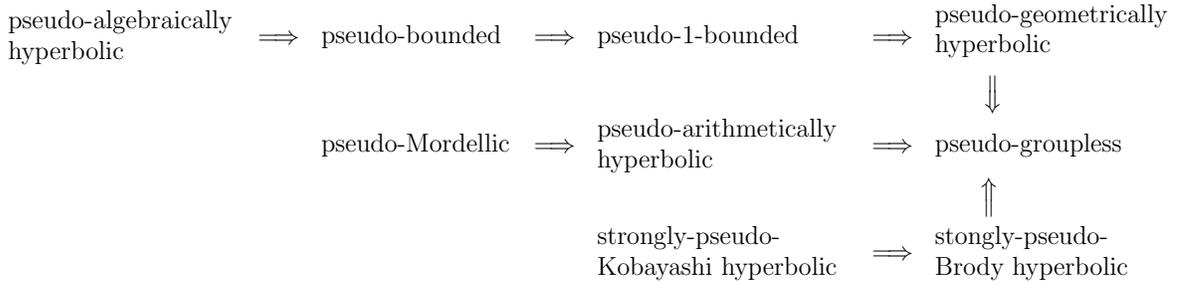
\begin{center}\scalebox{0.82}{
\begin{tabular}{m{3.7cm}cm{3.1cm}cm{4.1cm}cm{3.8cm}}
pseudo-algebraically hyperbolic &$\Longrightarrow$ &pseudo-bounded &$\Longrightarrow$ &pseudo-1-bounded &$\Longrightarrow$ &pseudo-geometrically hyperbolic \\[0.2cm]
&&&&&&\qquad \rotatebox[origin=c]{270}{$\Longrightarrow$}\\
&&pseudo-Mordellic &$\Longrightarrow$ &pseudo-arithmetically hyperbolic &$\Longrightarrow$ &pseudo-groupless \\
&&&&&&\qquad \rotatebox[origin=c]{90}{$\Longrightarrow$}\\[0.2cm]
&&&& strongly-pseudo-Kobayashi hyperbolic &$\Longrightarrow$ &stongly-pseudo-Brody hyperbolic
\end{tabular}}
\end{center}\mbox{}\\[0.2cm]

We stress that the Strong Lang-Vojta conjecture is concerned with classifying projective varieties of general type via their complex-analytic or arithmetic properties.  Recall that Campana's special varieties can be considered as being   opposite to varieties of general type. As Campana's conjectures are concerned with characterising special varieties via their complex-analytic or arithmetic properties, his conjectures should be  considered as     providing another part of the conjectural picture.  We refer the reader to \cite{CampanaBook} for a discussion of Campana's conjectures.
 
The following conjecture is a priori weaker then the Strong Lang-Vojta conjecture, as it is only concerned with hyperbolic varieties. It is not clear to us whether the Strong Lang-Vojta conjecture can be deduced from the following weaker version, as there are pseudo-hyperbolic projective varieties which are not birational to a hyperbolic projective variety.

 \begin{conjecture}[Weak Lang-Vojta, IV]\label{conj:weak} 
Let $X$ be a projective variety over $k$. Then the following statements are equivalent.
\begin{enumerate}
\item The variety $X$  is strongly-Brody hyperbolic over $k$.
\item The variety $X$ is strongly-Kobayashi hyperbolic over $k$.
\item The projective variety $X$ is Mordellic over $k$.
\item The projective variety $X$ is arithmetically hyperbolic over $k$.
\item The projective variety $X$ is  groupless over $k$.
\item The projective variety $X$ is algebraically hyperbolic over $k$.
\item The projective variety $X$ is bounded  over $k$.
\item The projective variety $X$ is  $1$-bounded over $k$.
\item The projective variety $X$ is geometrically hyperbolic over $k$.
\item Every integral subvariety of $X$ is of general type. 
\end{enumerate}
\end{conjecture}
 
 \begin{remark}[Strong implies Weak]
 Let us illustrate why the strong Lang-Vojta conjecture implies the Weak Lang-Vojta conjecture. To do so, let $X$ be a projective variety. Assume that $X$ is groupless. Then $X$ is pseudo-groupless. Thus, by the Strong Lang-Vojta conjecture, we have that $X$ is Mordellic modulo some proper closed subset $\Delta\subset X$. Now, since $X$ is groupless, it follows that $\Delta$ is groupless. Repeating the above argument shows that   $\Delta$ is Mordellic, so that $X$ is Mordellic.  
 \end{remark}

 We know \emph{more} about the Weak Lang-Vojta conjecture than we do about the Strong Lang-Vojta conjecture. Indeed, it is known that $(1)\iff (2)$ by Brody's Lemma . Also, it is not hard to show that  $(2)\implies (5)$. Moreover, we know that $(3)\implies (4)$ and $(4)\implies (5)$.
Of course, we also have that  $(6)\implies (7)$, $(7)\implies (8)$ and $(8)\iff (9)$. In addition, we also have that $(1)\implies (6)$ and that $(10)\implies (5)$. The following diagram summarizes these known implications. \\

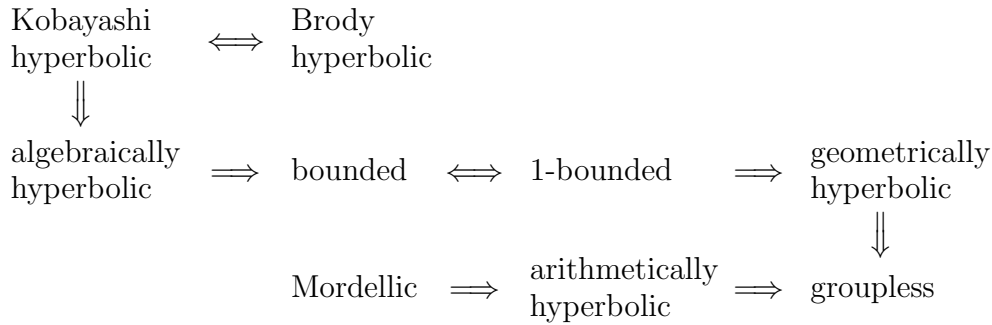
\begin{center}
\begin{tabular}{m{2.25cm}cm{1.7cm}cm{2.35cm}cm{2.0cm}c}
Kobayashi hyperbolic &$\Longleftrightarrow$ &Brody \mbox{hyperbolic} &\\[0.2cm]
\qquad \rotatebox[origin=c]{270}{$\Longrightarrow$}\\[0.2cm]
algebraically hyperbolic &$\Longrightarrow$ &bounded &$\Longleftrightarrow$ &1-bounded &$\Longrightarrow$ &geometrically hyperbolic &\\[0.2cm]
&&&&&&\qquad \rotatebox[origin=c]{270}{$\Longrightarrow$}\\[-0.1cm]
&&Mordellic &$\Longrightarrow$ & arithmetically hyperbolic &$\Longrightarrow$ &groupless
\end{tabular}
\end{center}\mbox{}\\[0.0cm]

 \newpage
 The figure below illustrates a projective variety which satisfies the Weak Lang-Vojta conjecture. The picture shows that this variety has infinitely many points valued in a number field (in orange), admits an entire curve (in blue), admits algebraic maps of increasing degree from some fixed  curve (in red), and admits a non-constant map from an abelian variety (in green). It is therefore a non-Mordellic, non-Brody hyperbolic, non-bounded, and non-groupless projective variety. 
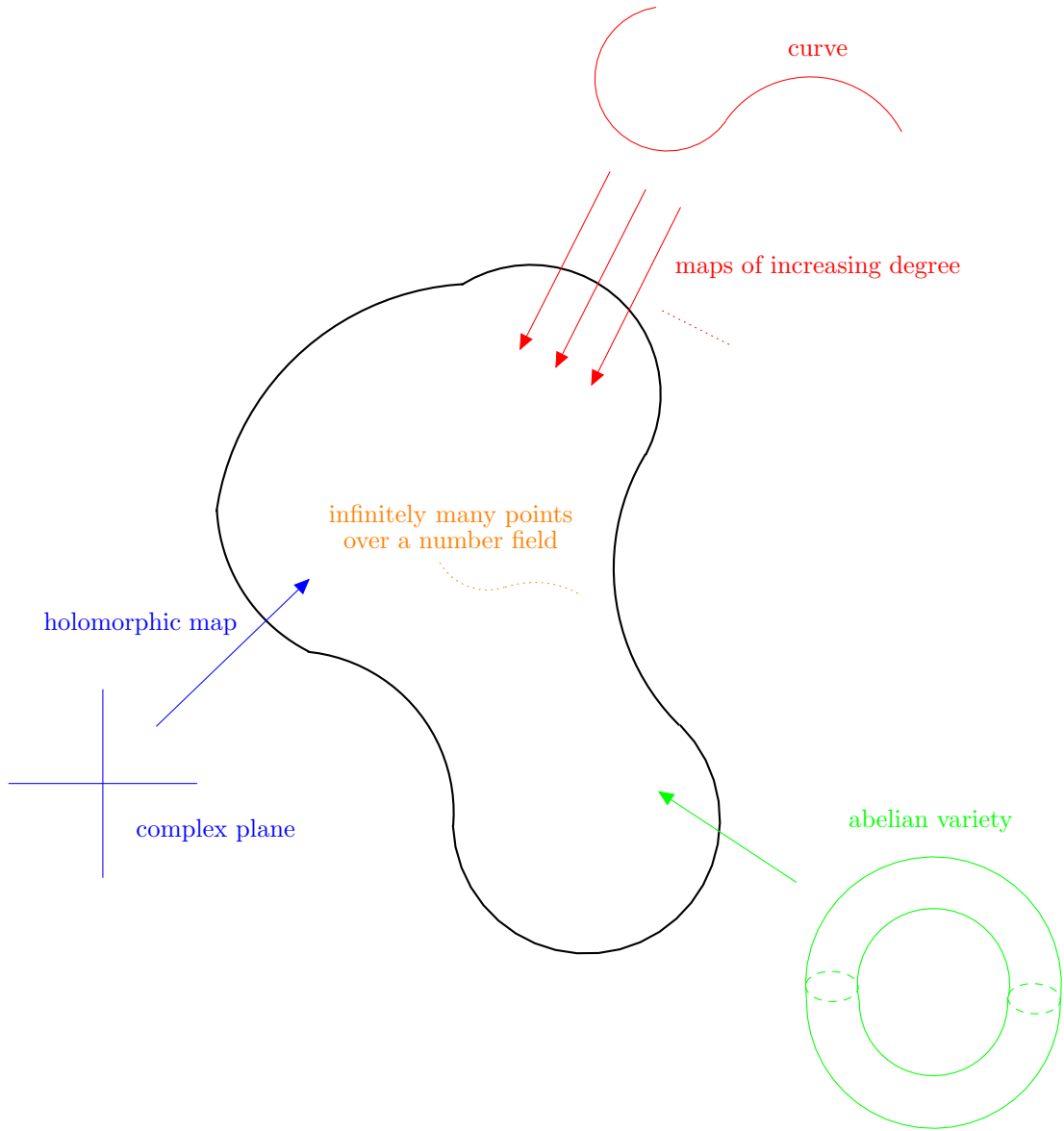
\vspace{2.5cm}

 \begin{center}
 \begin{tikzpicture}[line cap=round,line join=round,>=triangle 45,scale=1.3]
\draw [thick, shift={(0,0)}] plot[domain=1.63:3,variable=\t]({1*2.82*cos(\t r)+0*2.82*sin(\t r)},{0*2.82*cos(\t r)+1*2.82*sin(\t r)});
\draw [thick, shift={(-1,0.5)}] plot[domain=3.19:4.24,variable=\t]({1*1.79*cos(\t r)+0*1.79*sin(\t r)},{0*1.79*cos(\t r)+1*1.79*sin(\t r)});
\draw [thick, shift={(-1.99,-2.81)}] plot[domain=-0.08:1.47,variable=\t]({1*1.72*cos(\t r)+0*1.72*sin(\t r)},{0*1.72*cos(\t r)+1*1.72*sin(\t r)});
\draw [thick, shift={(1.14,-2.89)}] plot[domain=-3.1:0.79,variable=\t]({1*1.42*cos(\t r)+0*1.42*sin(\t r)},{0*1.42*cos(\t r)+1*1.42*sin(\t r)});
\draw [thick, shift={(3.78,-0.21)}] plot[domain=2.6:3.93,variable=\t]({1*2.35*cos(\t r)+0*2.35*sin(\t r)},{0*2.35*cos(\t r)+1*2.35*sin(\t r)});
\draw [thick, shift={(0.55,1.64)}] plot[domain=-0.48:2.13,variable=\t]({1*1.38*cos(\t r)+0*1.38*sin(\t r)},{0*1.38*cos(\t r)+1*1.38*sin(\t r)});
\draw[blue] (-3,-2.5)-- (-5,-2.5);
\draw[blue] (-4,-3.5)-- (-4,-1.5);
\draw [blue, ->] (-3.43,-1.89) -- (-1.8,-0.32);
\draw [red, shift={(2,5)}] plot[domain=1.73:5.66,variable=\t]({1*0.77*cos(\t r)+0*0.77*sin(\t r)},{0*0.77*cos(\t r)+1*0.77*sin(\t r)});
\draw [red, shift={(3.52,3.92)}] plot[domain=0.49:2.52,variable=\t]({1*1.1*cos(\t r)+0*1.1*sin(\t r)},{0*1.1*cos(\t r)+1*1.1*sin(\t r)});
\draw [red, ->] (1.39,4.01) -- (0.43,2.11);
\draw [red, ->] (1.77,3.82) -- (0.81,1.92);
\draw [red, ->] (2.14,3.63) -- (1.19,1.73);
\draw [green, shift={(4.83,-4.82)}] plot[domain=-3.17:0.06,variable=\t]({1*0.79*cos(\t r)+0*0.79*sin(\t r)},{0*0.79*cos(\t r)+1*0.79*sin(\t r)});
\draw [green, shift={(4.83,-4.82)}] plot[domain=-3.15:0.04,variable=\t]({1*1.35*cos(\t r)+0*1.35*sin(\t r)},{0*1.35*cos(\t r)+1*1.35*sin(\t r)});
\draw [green, shift={(4.83,-4.64)}] plot[domain=-0.16:3.34,variable=\t]({1*0.81*cos(\t r)+0*0.81*sin(\t r)},{0*0.81*cos(\t r)+1*0.81*sin(\t r)});
\draw [green, shift={(4.83,-4.64)}] plot[domain=-0.09:3.26,variable=\t]({1*1.36*cos(\t r)+0*1.36*sin(\t r)},{0*1.36*cos(\t r)+1*1.36*sin(\t r)});
\draw [green, rotate around={178.28:(3.75,-4.66)},dash pattern=on 3pt off 3pt] (3.75,-4.66) ellipse (0.28cm and 0.16cm);
\draw [green, ->] (3.37,-3.55) -- (1.9,-2.58);
\draw [orange, shift={(0.09,0.15)},dotted]  plot[domain=3.68:5.04,variable=\t]({1*0.59*cos(\t r)+0*0.59*sin(\t r)},{0*0.59*cos(\t r)+1*0.59*sin(\t r)});
\draw [orange, shift={(0.59,-1.33)},dotted]  plot[domain=1.08:1.9,variable=\t]({1*0.97*cos(\t r)+0*0.97*sin(\t r)},{0*0.97*cos(\t r)+1*0.97*sin(\t r)});
\draw [green, rotate around={176.71:(5.9,-4.79)},dash pattern=on 3pt off 3pt] (5.9,-4.79) ellipse (0.28cm and 0.16cm);
\draw [red, dotted] (1.95,2.53)-- (2.7,2.15);
\begin{footnotesize}
\draw[red] (3.6,5.3) node {curve};
\draw[blue] (-2.8,-3) node {complex plane};
\draw[green] (4.8, -2.9) node {abelian variety};
\draw[red] (3.6,3.0) node {maps of increasing degree};
\draw[orange] (-0.3,0.35) node {infinitely many points};
\draw[orange] (-0.3,0.1) node {over a number field};
\draw[blue] (-3.6,-0.8) node {holomorphic map};
\end{footnotesize}
\end{tikzpicture}
\end{center}

\newpage

  \subsection{The conjecture on exceptional loci}
  
  We now define the exceptional loci for every notion that we have seen so far.
   As usual, we let $k$ be  an algebraically closed field of characteristic zero.

 \begin{definition}     Let $X$ be a     variety over $k$. 
 \begin{itemize}
 \item We define $\Delta_X^{gr}$ to be the intersection of all proper closed subset $\Delta$ such that $X$ is groupless modulo $\Delta$. Note that  $\Delta_X^{gr}$ is a   closed subset of $X$ and that $X$ is groupless modulo $\Delta_X^{gr}$. We refer to $\Delta_X^{gr}$ as the \emph{groupless-exceptional locus} of $X$. 
 \item We define $\Delta_X^{ar-hyp}$ to be the intersection of all proper closed subsets $\Delta$ such that $X$ is arithmetically hyperbolic modulo $\Delta$. Note that $X$ is arithmetically hyperbolic modulo $\Delta_X^{ar-hyp}$. We refer to $\Delta_X^{ar-hyp}$ as the \emph{arithmetic-exceptional locus} of $X$. 
 \item 
We define $\Delta_X^{Mor}$ to be the intersection of all proper closed subsets $\Delta$ such that $X$ is Mordellic  modulo $\Delta$. Note that $X$ is Mordellic   modulo $\Delta_X^{Mor}$. We refer to $\Delta_X^{Mor}$ as the \emph{Mordellic-exceptional locus} of $X$. 
\end{itemize}
 \end{definition}
 
 Assuming $X$ is a proper variety over $k$ for a moment, it seems worthwhile stressing that $\Delta_X^{gr}$ equals the (Zariski) closure of   Lang's algebraic exceptional set $\mathrm{Exc}_{alg}(X)$ as defined in \cite[p.~160]{Lang2}. 
  
  \begin{definition}
Let $X$ be a variety over $\CC$.  
\begin{itemize}
\item We let $\Delta_X^{Br}$ be the intersection of all closed subsets $\Delta$ such that $X$ is Brody hyperbolic modulo $\Delta$. Note that $\Delta^{Br}_X$ is a closed subset of $X$ and that $X$ is Brody hyperbolic modulo $\Delta_X^{Br}$. We refer to $\Delta^{Br}_X$ as the \emph{Brody-exceptional locus} of $X$. 
\item  We let $\Delta^{Kob}_X$ be the intersection of all closed subsets $\Delta$ such that $X$ is Kobayashi hyperbolic modulo $\Delta$. Note that $\Delta^{Kob}_X$ is a closed subset of $X$ and that $X$ is Kobayashi hyperbolic modulo $\Delta^{Kob}_X$. We refer to $\Delta^{Kob}_X$ as the \emph{Kobayashi-exceptional locus} of $X$. 
\end{itemize}
\end{definition}  

We note that $\Delta_X^{Br}$ coincides with Lang's analytic exceptional set $\mathrm{Exc}(X)$ (defined in \cite[p.~160]{Lang2}). Indeed, $\mathrm{Exc}(X)$ is defined to be the Zariski closure of the union of all images of non-constant entire curves $\mathbb{C}\to X^{\an}$.

 \begin{definition} Let $X$ be a     projective scheme over $k$. 
 \begin{itemize}
 \item We define $\Delta_X^{alg-hyp}$ to be the intersection of all proper closed subsets $\Delta$ such that $X$ is algebraically hyperbolic modulo $\Delta$. Note that  $\Delta_X^{alg-hyp}$ is a proper closed subset of $X$ and that $X$ is algebraically hyperbolic modulo $\Delta^{alg-hyp}_X$. We refer to $\Delta_X^{alg-hyp}$ as the \emph{algebraic-exceptional locus} of $X$. 
 \item   For $n\geq 1$, we define $\Delta_X^{n-bounded}$ to be the intersection of all proper closed subsets $\Delta$ such that $X$ is $n$-bounded modulo $\Delta$. Note that  $\Delta_X^{n-bounded}$ is a proper closed subset of $X$ and that $X$ is $n$-bounded modulo $\Delta_X^{n-bounded}$. We refer to $\Delta_X^{n-bounded}$ as the \emph{$n$-bounded-exceptional locus} of $X$. 
 \item   We define $\Delta_X^{bounded}$ to be the intersection of all proper closed subsets $\Delta$ such that $X$ is bounded modulo $\Delta$. Note that  $\Delta_X^{bounded}$ is a proper closed subset of $X$ and that $X$ is bounded modulo $\Delta_X^{bounded}$. We refer to $\Delta_X^{bounded}$ as the \emph{bounded-exceptional locus} of $X$. 
\item   We define $\Delta_X^{geom-hyp}$ to be the intersection of all proper closed subsets $\Delta$ such that $X$ is geometrically hyperbolic modulo $\Delta$. Note that  $\Delta_X^{geom-hyp}$ is a proper closed subset of $X$ and that $X$ is geometrically hyperbolic modulo $\Delta_X^{geom-hyp}$. We refer to $\Delta_X^{geom-hyp}$ as the \emph{geometric-exceptional locus} of $X$. 
\end{itemize}
 \end{definition}
  
 The strongest version of Lang-Vojta's conjecture stated in these notes claims the equality of all exceptional loci. Note that these  loci are all, by definition, closed subsets. This is to be contrasted with Lang's definition of his ``algebraic exceptional set'' (see \cite[p.~160]{Lang2}).

\begin{conjecture}[Strongest Lang-Vojta conjecture]\label{conj:strongest} Let $k$ be an algebraically closed field of characteristic zero.
Let $X$ be a projective variety over $k$. Then the following three statements hold.
\begin{enumerate} 
\item  We have that \[
\Delta^{gr}_X =   \Delta^{Mor}_X = \Delta^{geom-hyp}_X = \Delta^{1-bounded}_X = \Delta^{bounded}_X = \Delta^{alg-hyp}_X.
\] 
\item The projective variety $X$ is of general type if and only if $\Delta_X^{gr} \neq X$.
\item If $k=\CC$, then $\Delta_X^{gr} = \Delta_X^{Br} = \Delta_X^{Kob}$.
\end{enumerate} 
\end{conjecture}

 \begin{remark}[Which inclusions do we know?] Let $X$ be a projective scheme over $k$.
We have that
\[ \Delta_X^{gr}\subset \Delta_X^{ar-hyp} \subset \Delta_X^{Mor},\] and \[ \Delta_X^{gr}\subset \Delta_X^{geom-hyp}\subset \Delta_X^{1-bounded}\subset \Delta_X^{bounded}\subset \Delta_X^{alg-hyp}.
 \] If $k$ is uncountable, then \[\Delta_X^{1-bounded} = \Delta_X^{bounded}.\] If $k=\CC$, then  \[
 \Delta^{gr}_X \subset \Delta^{Br}_X \subset \Delta^{Kob}_X.\] 
 \end{remark}
 
 \begin{remark}[Reformulating Brody's lemma]
 It is not known whether $\Delta^{Kob}_X \subset \Delta^{Br}_X$.   Brody's lemma can be stated as saying that, if $\Delta^{Br}_X$ is empty, then $\Delta^{Kob}_X$ is empty. 
\end{remark}

\begin{remark}[Reformulating Demailly's theorem]
It is   not known whether  $\Delta^{alg-hyp}_X \subset \Delta^{Kob}_X$. Demailly's theorem (Theorem \ref{thm:demailly1}) can be stated as saying that, if $\Delta^{Kob} $ is empty, then $\Delta^{alg-hyp}$ is empty.
\end{remark}

  \section{Closed subvarieties of abelian varieties}\label{section:questions}
  We have gradually worked our way towards the following theorem which   says that the Strongest Lang-Vojta conjecture holds for closed subvarieties of abelian varieties.    Recall that, for $X$ a closed subvariety of an abelian variety $A$, the subset $\mathrm{Sp}(X)$ is defined to be the union of translates of positive-dimensional abelian subvarieties of $A$ contained in $A$. It is a fundamental fact that $\mathrm{Sp}(X)$ is a closed subset of $X$. It turns out that $\mathrm{Sp}(X)$ is the ``exceptional locus'' of $X$ in any sense of the word ``exceptional locus''.
  
 \begin{theorem}[Bloch-Ochiai-Kawamata, Faltings, Yamanoi, Kawamata-Ueno]\label{thm:exc} Let $A$ be an abelian variety over $k$, and let $X\subset A$ be a closed subvariety. Then the following statements hold.
 \begin{enumerate}
 \item We have that $\mathrm{Sp}(X) \neq X$ of $X$ equals   if and only if $X$ is   of general type.
 \item We have that  $$\mathrm{Sp}(X) = \Delta_X^{gr}=\Delta_X^{Mor} = \Delta_X^{ar-hyp}      = \Delta_X^{geom-hyp} = \Delta_X^{1-bounded} = \Delta_X^{bounded} = \Delta_X^{alg-hyp}.$$ 
 \item If $k=\CC$, then $\Delta^{gr}_X = \Delta_X^{Br} = \Delta_X^{Kob}$.
 \end{enumerate}
 \end{theorem}
 \begin{proof}  The fact that $\mathrm{Sp}(X) \neq X$ if and only if $X$ is of general type is due to Kawamata-Ueno (see also Theorem \ref{thm:ku}). Moreover, 
  an elementary argument (see Example \ref{example:groupless}) shows that $X$ is groupless modulo $\mathrm{Sp}(X)$, so that $\Delta^{gr}_X\subset \mathrm{Sp}(X)$. On the other hand, it is clear from the definition that  $\mathrm{Sp}(X) \subset \Delta^{gr}_X$. This shows that $\mathrm{Sp}(X) = \Delta_X^{gr}$.
 
 By Faltings's theorem (Theorem \ref{thm:faltings_big}), we have that $X$ is Mordellic modulo $\mathrm{Sp}(X)$. This shows that $\Delta^{Mor}_X = \Delta_X^{ar-hyp} = \Delta^{gr}_X = \mathrm{Sp}(X)$.  (One can also show that $\Delta^{ar-hyp}_X = \Delta^{Mor}_X$ without appealing to Faltings's theorem. Indeed, as $X$ is a closed subvariety of an abelian variety, it follows from Corollary \ref{cor:psar_is_mor}  that $X$ is arithmetically hyperbolic modulo $\Delta$ if and only if $X$ is Mordellic modulo $\Delta$.)
 
 It follows from Bloch--Ochiai--Kawamata's theorem that $\Delta^{Br}_X = \mathrm{Sp}(X)$. Yamanoi improved this result   and showed that $\Delta^{Kob}_X = \mathrm{Sp}(X)$; see Theorem \ref{thm:yamanoi_Kob} (or the original \cite[Theorem~1.2]{Yamanoi2}). In his earlier work \cite[Corollary~1.(3)]{Yamanoi1}, Yamanoi proved that $\Delta^{alg-hyp}_X = \mathrm{Sp}(X)$. Since \[\Delta^{geom-hyp}_X\subset \Delta^{1-bounded}_X \subset \Delta_X^{bounded} \subset \Delta^{alg-hyp}_X, \] this concludes the proof. 
 \end{proof}

\section{Evidence for Lang-Vojta's conjecture}\label{section:evidence}
In the previous sections, we  defined every notion appearing in Lang-Vojta's conjecture, and  we stated the   ``Strongest'', ``Stronger'' and ``Weakest'' versions of Lang-Vojta's conjectures. We also indicated the known implications between these notions, and that the Strongest Lang-Vojta conjecture is known to hold for      closed subvarieties of abelian varieties by work of Bloch-Ochiai-Kawamata, Faltings, Kawamata-Ueno, and Yamanoi. 

In the following four sections, we will present  some   evidence for Lang-Vojta's conjectures. The results in the following sections are all in accordance with the Lang-Vojta conjectures.

\section{Dominant rational self-maps of pseudo-hyperbolic varieties}\label{section:bir}

     Let us start with a classical finiteness result of Matsumura  \cite[\S11]{Iitaka}.

\begin{theorem}[Matsumura]\label{thm:mat}
If $X$ is  a proper integral variety of general type over $k$, then  the set of dominant rational self-maps $X\dashrightarrow X$  is finite.
\end{theorem}

Note that Matsumura's theorem is a vast generalization of the statement that a smooth curve of genus at least two has only finitely many automorphisms. Motivated by Lang-Vojta's conjecture, 
the arithmetic analogue of Matsumura's theorem is proven in \cite{JXie} (building on the results in \cite{JAut}) and can be stated as follows. 
    
    \begin{theorem}\label{thm:bir_is_finite} If $X$ is a proper pseudo-Mordellic integral variety over $k$, then the set of rational dominant self-maps $X\dashrightarrow X$ is finite.
    \end{theorem}
    \begin{proof}[Idea of proof]

We briefly indicate three ingredients of   the proof of Theorem \ref{thm:bir_is_finite}.  
\begin{enumerate}
\item First, one can use Amerik's theorem on dynamical systems \cite{Amerik2011} to show that every dominant rational self-map is a birational self-map of finite order whenever $X$ is a pseudo-Mordellic projective variety.
\item  One can show that, if $X$ is a projective integral variety over $k$ such that $\Aut_k(X)$ is infinite, then $\Aut_k(X)$ has an element of infinite order. (It is crucial here that $k$ is of characteristic zero.) This result is proven in \cite{JAut}.  
\item If $X$ is a projective non-uniruled integral variety over $k$ such that  $\mathrm{Bir}_k(X)$ is infinite, then $\mathrm{Bir}_k(X)$ has a point of infinite order. To prove this, one can  use   Prokhorov-Shramov's notion  of quasi-minimal models (see \cite{ProShramov2014})  to reduce to  the analogous finiteness result for automorphisms stated in $(2)$. The details are in \cite{JXie}.
\end{enumerate}
Combining $(1)$ and $(3)$, one obtains the desired result for pseudo-Mordellic projective varieties (Theorem \ref{thm:bir_is_finite}).
 \end{proof}

    There is a similar finiteness statement for pseudo-algebraically hyperbolic varieties. This finiteness result is proven in \cite{JKa} for algebraically hyperbolic varieties, and in \cite{JXie} for pseudo-algebraically hyperbolic varieties. 
    
    \begin{theorem}
    If $X$ is a projective pseudo-algebraically hyperbolic integral variety over $k$, then  the set of dominant rational self-maps $X\dashrightarrow X$ is finite.
    \end{theorem}
    
In fact, more generally, we have the following a priori stronger result. 
    
    \begin{theorem}\label{thm:bir_is_finite_bounded}     If $X$ is a projective pseudo-$1$-bounded   integral variety over $k$, then  the set of dominant rational self-maps $X\dashrightarrow X$ is finite.
    \end{theorem}
    \begin{proof}
For $1$-bounded varieties this is proven in \cite{JKa}.  The more general statement for pseudo-$1$-bounded varieties is proven in \cite{JXie} by combining Amerik's theorem \cite{Amerik2011} and Prokhorov-Shramov's theory of quasi-minimal models \cite{ProShramov2014} with Weil's Regularization Theorem and properties of dynamical degrees of rational dominant self-maps.
    \end{proof}

As the reader may have noticed,   for pseudo-Mordellic,  pseudo-algebraically hyperbolic  and pseudo-$1$-bounded projective varieties we have  satisfying  results. 

What do we know in the complex-analytic setting?  We have the following result of  Noguchi \cite[Theorem~5.4.4]{KobayashiBook} for Brody hyperbolic varieties.

    \begin{theorem}[Noguchi]\label{thm:noguchii}
    If $X$ is a Brody hyperbolic projective integral variety over $\CC$, then $\mathrm{Bir}_{\CC}(X)$ is finite.
    \end{theorem}
    \begin{proof}[First proof of Theorem \ref{thm:noguchii}]
    Since a Brody hyperbolic projective integral variety over $\CC$ is bounded by, for instance, Demailly's theorem (Theorem  \ref{thm:demailly1}), this follows from Theorem \ref{thm:bir_is_finite_bounded}.  
    \end{proof}
    
      \begin{proof}[Second proof of Theorem \ref{thm:noguchii}]
    Let $Y\to X$ be a resolution of singularities of $X$. Note that, every birational morphism $X\dashrightarrow X$ induces a dominant rational map $Y\dashrightarrow X$. Since $X$ has no rational curves (as $X$ is Brody hyperbolic) and $Y$ is smooth, by \cite[Lemma~3.5]{JKa}, the rational map $Y\dashrightarrow X$ extends uniquely to a surjective morphism $Y\to X$.  
    
    Therefore, we have that 
    \[
    \mathrm{Bir}_{\CC}(X) \subset \mathrm{Sur}_{\CC}(Y,X)
    \]
    Noguchi  proved   that the latter set is finite (see Theorem \ref{thm:noguchi} below). He does so by showing that it is the set of $\CC$-points on a finite type zero-dimensional scheme over $\CC$. We discuss this result of Noguchi  in more detail in the next section.
    \end{proof}
    
 It is important to note that,    in light of Green-Griffiths' and Lang-Vojta's conjectures, one expects an analogous  finiteness result for pseudo-Brody hyperbolic varieties (as pseudo-Brody hyperbolic varieties should be of general type).   This     is however not known, and we state it as a separate conjecture.
 
 \begin{conjecture}[Pseudo-Noguchi, I]\label{conj:noguchi}
 If $X$ is a pseudo-Brody hyperbolic projective integral variety over $\CC$, then $\mathrm{Bir}_{\CC}(X)$ is finite.
 \end{conjecture}

    \begin{remark}[What do we not know yet?]
 First, it is not known whether the automorphism group of a groupless projective variety is finite.  Also,    it is not known whether a pseudo-Kobayashi hyperbolic projective variety has a finite automorphism group. Moreover, it is not know whether a geometric hyperbolic projective variety has only finitely many automorphisms. As these problems are unresolved, the finiteness of the set of birational self-maps is also still open.
    \end{remark}

\section{Finiteness of moduli spaces of surjective morphisms}\label{section:sur}
Our starting point in this section is the following finiteness theorem of Noguchi for dominant rational maps from a fixed variety to a hyperbolic variety (\emph{formerly} a conjecture of Lang); see \cite[\S 6.6]{KobayashiBook} for a discussion of the history of this result.

\begin{theorem}[Noguchi]\label{thm:noguchi}
If $X$  is a Brody hyperbolic proper variety over $\CC$ and $Y$ is a projective integral variety over $\CC$, then the set of dominant rational maps $f:Y\dashrightarrow X$ is finite.
\end{theorem}

In light of Lang-Vojta's conjecture, any ``hyperbolic'' variety should satisfy a similar finiteness property. In particular, one should expect similar (hence more general) results for bounded varieties, and such    results  are  obtained in \cite{JKa} over arbitrary algebraically closed fields $k$ of characteristic zero.
\begin{theorem} 
If $X$  is a $1$-bounded projective   variety over $k$ and $Y$ is a projective integral variety over $k$, then the set of dominant rational maps $f:Y\dashrightarrow X$ is finite.
\end{theorem}

In particular, the same finiteness statement holds for bounded varieties and algebraically hyperbolic varieties.  Indeed, such varieties are (obviously) $1$-bounded. 

\begin{corollary}
If $X$ is a bounded projective variety over $k$ (e.g., algebraically hyperbolic variety over $k$) and $Y$ is a projective integral variety over $k$, then the set of dominant rational maps $f:Y\dashrightarrow X$ is finite.
\end{corollary}

We now make a ``pseudo''-turn. In fact,   the finiteness result of Noguchi should actually hold under the weaker assumption that $X$ is only pseudo-Brody hyperbolic.   To explain this, recall that Kobayashi--Ochiai proved a finiteness theorem  for dominant rational maps from a given variety $Y$ to a fixed variety of general type $X$ which generalizes Matsumura's finiteness theorem for the group $\mathrm{Bir}_k(X)$ (Theorem \ref{thm:mat}).
 
\begin{theorem}[Kobayashi-Ochiai]
Let $X$ be a projective   variety over $k$ of general type. Then, for every projective integral variety $Y$, the set of dominant rational maps $f:Y\dashrightarrow X$ is finite.
\end{theorem}

In light of Lang-Vojta's conjectures and Kobayashi-Ochiai's theorem, any ``pseudo-hyperbolic'' variety should satisfy a similar finiteness property.  For example, Lang-Vojta's conjecture predicts a similar finiteness statement for  pseudo-Brody hyperbolic  projective varieties. We state this as a conjecture. Note that this conjecture is the ``pseudo''-version of Noguchi's theorem (Theorem \ref{thm:noguchi}), and clearly implies Conjecture \ref{conj:noguchi}.

\begin{conjecture}[Pseudo-Noguchi, II]
If $X$  is a \text{pseudo}-Brody hyperbolic proper variety over $\CC$ and $Y$ is a projective integral variety over $\CC$, then the set of dominant rational maps $f:Y\dashrightarrow X$ is finite.
\end{conjecture}

Now, as any ``pseudo-hyperbolic''  variety is pseudo-groupless, it is natural to first try and see what one can say about pseudo-groupless varieties. For simplicity, we will focus on surjective morphisms (as opposed to dominant rational maps) in the rest of this section.

There is a standard approach to establishing the finiteness of the set of surjective morphisms from one projective scheme to another. To explain this, let us recall some notation from Section \ref{section:boundedness}. Namely, if $X$ and $Y$ are projective schemes over $k$, we let $\underline{\mathrm{Hom}}_k(Y,X)$ be the scheme parametrizing morphisms $X\to Y$. Note that $\underline{\Hom}_k(Y,X)$ is a countable disjoint union of  quasi-projective schemes over $k$.   Moreover, we let $\underline{\mathrm{Sur}}_k(Y,X)$ be the scheme parametrizing surjective morphisms $Y\to X$, and note that $\underline{\mathrm{Sur}}_k(Y,X)$ is a closed subscheme of $\underline{\Hom}_k(Y,X)$.

The standard approach to establishing the finiteness of  the set $\mathrm{Sur}_k(Y,X)$ is to interpret it as the set of $k$-points on the scheme $\underline{\mathrm{Sur}}_k(Y,X)$. This makes it tangible to techniques from deformation theory. Indeed, to show that $\mathrm{Sur}_k(Y,X)$ is finite, it suffices to establish the following two statements:
\begin{enumerate}
\item  The tangent space to each  point of $\underline{\mathrm{Sur}}_k(Y,X)$ is trivial;
\item  The scheme $\underline{\mathrm{Sur}}_k(Y,X)$ has only finitely many connected components.
 \end{enumerate}
 
 It is   common to refer to the first statement as a rigidity statement, as it boils down to showing that the objects parametrized by $\underline{\mathrm{Sur}}_k(Y,X)$ are infinitesimally rigid. Also, it is   standard to refer to the second statement as being a boundedness property. For example, if $Y$ and $X$ are curves and $X$ is of genus at least two, the finiteness of $\mathrm{Sur}_k(Y,X)$ is proven precisely in this manner; see \cite[\S II.8]{Maz86}. We refer the reader to \cite{KovacsSubs} for a further discussion of the rigidity/boundedness approach to proving finiteness results for other moduli spaces.
 
We now focus  on the rigidity of surjective morphisms $Y\to X$. The  following rigidity theorem for pseudo-groupless varieties will prove to be extremely useful.  
This result is a consequence of a much more general statement about the   deformation space of a surjective morphism due to Hwang--Kebekus--Peternell \cite{HKP}.

\begin{theorem}[Hwang-Kebekus-Peternell + $\epsilon$]\label{thm:ps_grp_sur}
If $Y$ is a projective normal variety over $k$ and $X$ is a pseudo-groupless projective variety over $k$, then the scheme $\underline{\mathrm{Sur}}_k(Y,X)$ is a countable disjoint union of zero-dimensional smooth projective schemes over $k$.
\end{theorem}
\begin{proof} As is shown in  \cite{JXie}, this is a consequence of Hwang--Kebekus--Peternell's result on the infinitesimal deformations of a surjective morphism $Y\to X$. Indeed, since $X$ is non-uniruled (Remark \ref{remark:non_uni}), for every such surjective morphism $f:Y\to X$, there is a finite morphism $Z\to X$ and a morphism $Y\to Z$  such that $f$ is the composed map $Y\to Z\to X$. Moreover, the identity component $ \mathrm{Aut}_{Z/k}^0$ of the automorphism group scheme surjects onto the connected component of $f$ in $\underline{\Hom}_k(Y,X)$. Since $X$ is pseudo-groupless, the same holds for $Z$. It is then not hard to verify that $ \Aut^0_{Z/k}$ is trivial, so that the connected component of $f$ in $\underline{\Hom}_k(Y,X)$ is trivial.
\end{proof}

\begin{remark}
There are projective varieties $X$ which are \textbf{not} pseudo-groupless over $k$, but for which the conclusion of the theorem above still holds. For example, a K3 surface or the  blow-up of a simple abelian surface $A$ in its origin.   This means that the rigidity of surjective morphisms follows from properties strictly \emph{weaker} than pseudo-hyperbolicity. We refer to \cite{JXie} for a more general statement concerning  rigidity of surjective morphisms.
\end{remark}
 
 When introducing the notions appearing in Lang-Vojta's conjecture, we made sure to emphasize that every one of these is pseudo-groupless. Thus, roughly speaking, any property we prove for pseudo-groupless varieties holds for all pseudo-hyperbolic varieties.  This gives us the following rigidity statement.
 
 \begin{corollary}[Rigidity for pseudo-hyperbolic varieties]
  Let $X$  be a projective integral variety over $k$ and let $Y$ be a projective normal variety over $k$. Assume that one of the following statements holds.
  \begin{enumerate}
  \item The   variety $X$ is pseudo-groupless over $k$.
  \item The   variety $X$ is pseudo-Mordellic over $k$.
  \item The projective variety $X$ is pseudo-algebraically hyperbolic over $k$.
  \item The projective variety $X$ is pseudo-$1$-bounded over $k$.
  \item The projective variety $X$ is pseudo-bounded over $k$.
  \item The   variety $X$ is pseudo-geometrically hyperbolic over $k$.
  \item The field $k$ equals $\CC$ and    $X$ is pseudo-Brody hyperbolic.
  \end{enumerate} 
Then the scheme $\underline{\mathrm{Sur}}_k(Y,X)$ is a countable disjoint union of zero-dimensional smooth projective schemes over $k$.
 \end{corollary}

 \begin{proof} Assume that either $(1), (2), (3), (4), (5), (6),$ or $ (7), $   holds. Then $X$ is pseudo-groupless (as explained throughout these notes), so that the result  follows from Theorem \ref{thm:ps_grp_sur}.
 \end{proof}

Proving the finiteness of $\mathrm{Sur}_k(Y,X)$ or, equivalently, the  boundedness of $\underline{\mathrm{Sur}}_k(Y,X)$, for $X$ pseudo-groupless or pseudo-Mordellic seems to be out of reach currently. However,  for pseudo-algebraically hyperbolic varieties the desired finiteness property is proven in  \cite{JXie} and reads as follows.

\begin{theorem}
If $X$  is a  pseudo-algebraically  hyperbolic projective variety over $k$ and $Y$ is a projective integral variety over $k$, then the set of surjective morphisms $f:Y\to X$ is finite. \end{theorem}
 
 A similar result can be obtained for pseudo-bounded varieties. The precise result can be stated as follows.

\begin{theorem}
If $X$  is a  pseudo-bounded    projective variety over $k$ and $Y$ is a projective integral variety over $k$, then the set of surjective morphisms $f:Y\to X$ is finite. \end{theorem}

To prove the analogous finiteness property for pseudo-$1$-bounded varieties, we require (as in the previous section) an additional uncountability assumption on the base field.

 \begin{theorem} \label{thm:noguchi1} Assume $k$ is uncountable. 
 If $X$  is a pseudo-$1$-bounded  projective variety over $k$ and $Y$ is a projective integral variety over $k$, then the set of surjective morphisms   $f:Y\to X$ is finite.
 \end{theorem}

We conclude with the following finiteness result for pseudo-algebraically hyperbolic varieties. It is proven in \cite{JXie} using (essentially) the results in this section and the fact that pseudo-algebraically hyperbolic varieties are pseudo-geometrically hyperbolic.
 \begin{theorem}\label{thm:general}
 If $X$ is algebraically hyperbolic  modulo $\Delta$ over $k$, then  for every connected reduced projective variety $Y$ over $k$, every non-empty closed reduced subset $B\subset Y$, and every reduced closed subset $A\subset X$ not contained in $\Delta$, the set of morphisms $f:Y\to X$ with $f(B) = A$ is finite. 
 \end{theorem}
 
 Note that Theorem \ref{thm:general} can be applied with $B$ a point or $B=Y$. This shows that the statement generalizes the finiteness result of this section.

\section{Hyperbolicity along field extensions}\label{section:persistence}  
In this section we  study how different notions of pseudo-hyperbolicity appearing in Lang-Vojta's conjectures (except for those that only make sense over $\CC$ a priori) behave under field extensions. In other words, we study how the exceptional locus for each notion of hyperbolicity introduced in Section \ref{section:conjectures} behaves under field extensions.

Let us start with $X$  a variety of general type over a field $k$, and let $k\subset L$  be a field extension. It is natural to wonder whether $X_L$ is  also of general type over $L$. A simple argument comparing the spaces of global sections of $\omega_{X/k}$ and $\omega_{X_L/L}$ shows that this is indeed the case. This   observation is our starting point in this section. Indeed, the mere fact that varieties of general type remain varieties of general type after a field extension can be paired with the Strong Lang-Vojta conjecture to see  that similar statements should hold for pseudo-groupless varieties, pseudo-Mordellic varieties, and so on. 

The first three results we state in this section say that  this ``base-change'' property can be proven in some cases. For proofs we refer to  \cite{vBJK, JKa, JXie}.

\begin{theorem}\label{thm:persss} Let $k\subset L$ be an extension of algebraically closed fields of characteristic zero. Let $X$ be a projective scheme over $k$ and let $\Delta$ be a closed subset of $X$. Then the following statements hold.
\begin{enumerate}
\item If $X$ is of general type over $k$, then $X_L$ is of general type over $L$.
\item If $X$ is groupless modulo $\Delta$, then $X_L$ is groupless modulo $\Delta_L$.
\item If $X$ is algebraically hyperbolic modulo $\Delta$, then $X_L$ is algebraically hyperbolic modulo $\Delta_L$.
\item If $X$ is bounded modulo $\Delta$, then $X_L$ is bounded modulo $\Delta_L$.
\end{enumerate}
\end{theorem}

In this theorem we are missing (among others) the notions of $1$-boundedness and geometric hyperbolicity. In this direction we have the following result; see  \cite{vBJK, JLevin}. 

\begin{theorem}  Let $k\subset L$ be an extension of  \textbf{uncountable} algebraically closed fields of characteristic zero. Let $X$ be a projective scheme over $k$ and let $\Delta$ be a closed subset of $X$. Then the following statements hold.
\begin{enumerate}
\item If $X$ is $1$-bounded modulo $\Delta$, then $X_L$ is bounded modulo $\Delta_L$.
\item If $X$ is geometrically hyperbolic modulo $\Delta$, then $X_L$ is geometrically hyperbolic modulo $\Delta_L$.
\end{enumerate}
\end{theorem}

If $\Delta=\emptyset$, then we do not need to impose uncountability. 

\begin{theorem}  Let $k\subset L$ be an extension of    algebraically closed fields of characteristic zero. Let $X$ be a projective scheme over $k$ and let $\Delta$ be a closed subset of $X$. Then the following statements hold.
\begin{enumerate}
\item If $X$ is $1$-bounded, then $X_L$ is bounded.
\item If $X$ is geometrically hyperbolic, then $X_L$ is geometrically hyperbolic.
\end{enumerate}
\end{theorem}

The reader will have noticed the absence of the notion of Mordellicity and arithmetic hyperbolicity above.  The question of whether an arithmetically hyperbolic variety over $\Qbar$ remains arithmetically hyperbolic over a larger field is not an easy one in general, as should be clear from the following remark.

\begin{remark}[Persistence of arithmetic hyperbolicity] Let $f_1,\ldots, f_r\in \ZZ[x_1,\ldots,x_n]$ be polynomials, and let $X:=Z(f_1,\ldots, f_r)  = \Spec(\Qbar[x_1,\ldots,x_n]/(f_1,\ldots,f_r) \subset \mathbb{A}^n_{\Qbar}$ be the associated affine variety over $\Qbar$. To say that $X$ is arithmetically hyperbolic over $\Qbar$ is to say that, for every \emph{number field} $K$ and every finite set of finite places $S$ of $K$, the set of $a=(a_1,\ldots,a_n)\in \OO_{K,S}^n$ with $f_1(a) = \ldots = f_r(a) =0$ is finite. On the other hand, to say that $X_{\CC}$ is arithmetically hyperbolic over $\CC$ is to say that, for every \emph{$\ZZ$-finitely generated} subring $A\subset \CC$, the set of $a  = (a_1,\ldots, a_n)\in A^n$ with $f_1(a) = \ldots = f_r(a)=0$ is finite.   
\end{remark}

Despite the apparent difference between being arithmetically hyperbolic over $\Qbar$ and being arithmetically hyperbolic over $\CC$, it seems reasonable to suspect their equivalence. For $X$ a  projective variety, the following conjecture is a consequence of the Weak Lang-Vojta conjecture for $X$. However,  as it also seems reasonable in the non-projective case, we state it in this more generality. 

\begin{conjecture}[Persistence Conjecture] Let $k\subset L$ be an extension of algebraically closed fields of characteristic zero.
Let $X$ be a variety over $k$ and let $\Delta$ be a closed subset of $X$. If $X$ is arithmetically hyperbolic modulo $\Delta$ over $k$, then $X_L$ is arithmetically hyperbolic modulo $\Delta_L$ over $L$.
\end{conjecture}

Note that we will focus throughout on arithmetic hyperbolicity (as opposed to Mordellicity) as  its persistence along field extensions is easier to study. The reader may recall that the difference between Mordellicity and arithmetic hyperbolicity disappears for many varieties (e.g., affine varieties); see Section \ref{section:ps_mordell2} for a discussion. 

This conjecture is investigated in \cite{vBJK, JAut, JLevin, JLitt}.  As a basic example, the reader may note that Faltings proved that a smooth projective connected curve of genus at least two over $\Qbar$ is arithmetically hyperbolic over $\Qbar$ in \cite{Faltings2}. He then later explained in \cite{FaltingsComplements} that Grauert-Manin's function field version of the Mordell conjecture can be used to prove that a smooth projective connected curve of genus at least two over $k$ is arithmetically hyperbolic over $k$.

In the rest of this section, we will present some     results on the Persistence Conjecture. We start with the following result. 

\begin{theorem}\label{thm:A} Let $k\subset L$ be an extension of algebraically closed fields of characteristic zero. Let $X$ be an arithmetically hyperbolic variety over $k$ such that $X_L$ is geometrically hyperbolic over $L$. Then $X_L$ is arithmetically hyperbolic over $L$.
\end{theorem}

Note that Theorem \ref{thm:A} implies that the Persistence Conjecture holds for varieties over $k$ which are geometrically hyperbolic over any field extension of $L$.

  Theorem \ref{thm:A} is inspired by Martin-Deschamps's proof of the arithmetic Shafarevich conjecture over finitely generated fields (see also Remark \ref{remark:md}). Indeed, in Szpiro's seminar \cite{Szpiroa}, Martin-Deschamps gave a proof of the arithmetic Shafarevich conjecture   by using a specialization argument on the moduli stack of principally polarized abelian schemes; see \cite{Martin}. This specialization argument resides on Faltings's theorem  that the moduli space of principally polarized abelian varieties of fixed dimension over $\CC$ is geometrically hyperbolic over $\CC$. We note that Theorem \ref{thm:A} is essentially implicit in her line of reasoning.  

We will present applications of Theorem \ref{thm:A} to the Persistence Conjecture based on the results obtained  in    \cite{JAut, JLitt}.  However, before we give these applications, we mention the following result which implies that the Persistence Conjecture holds  for normal projective surfaces with non-zero irregularity $\mathrm{h}^1(X,\mathcal{O}_X)$.

\begin{theorem}\label{thm:B}
Let $X$ be a  projective surface over $k$ which admits a non-constant morphism to some abelian variety over $k$. Then $X$ is arithmetically hyperbolic over $k$ if and only if $X_L$ is arithmetically hyperbolic over $L$.
\end{theorem}

The proofs of Theorems \ref{thm:A} and \ref{thm:B} differ tremendously in spirit. In fact, we can not prove Theorem \ref{thm:B} by appealing to the geometric hyperbolicity of $X$ (as it is currently not known whether an arithmetically hyperbolic projective surface which admits a  non-constant map to an abelian variety is geometrically hyperbolic).  Instead, Theorem \ref{thm:B} is proven by appealing to the ``mild boundedness'' of abelian varieties; see \cite{vBJK}. More explicitly: in the proof of Theorem \ref{thm:B}, we use that, for every smooth connected curve $C$ over $k$, there exists an integer $n>0$ and  points $c_1,\ldots, c_n$ in $C(k)$ such that, for every abelian variety $A$ over $k$ and every $a_1,\ldots, a_n$ in $A(k)$, the set 
$$
\Hom_k((C,c_1,\ldots,c_n),(A,a_1,\ldots,a_n))
$$ is finite. This finiteness property for abelian varieties can be combined with the arithmetic hyperbolicity of the surface $X$ in Theorem \ref{thm:B} to show that the surface $X$ is mildly bounded. The  property of being mildly bounded is clearly \emph{much} weaker than being geometrically hyperbolic, but it turns out to be enough to show the Persistence Conjecture; see \cite[\S4.1]{JAut}. Note that it is a bit surprising that abelian varieties (as   they are very far from being hyperbolic) satisfy some ``mild'' version of geometric hyperbolicity. We refer the reader to \cite[\S 4]{JAut} for the definition of what this notion entails, and to \cite{vBJK} for the fact that abelian varieties are mildly bounded.  

We now focus as promised on  the applications of Theorem \ref{thm:A}. Our first application says that the Persistence Conjecture holds for all algebraically hyperbolic projective varieties.

\begin{theorem}  
Let $X$ be a projective algebraically hyperbolic variety over $k$. Then $X$ is arithmetically hyperbolic over $k$ if and only if, for every algebraically closed field $L$ containing $k$, the variety $X_L$ is arithmetically hyperbolic over $L$.
\end{theorem}

\begin{proof}
Since $X$ is algebraically hyperbolic over $k$, it follows from $(3)$ in  Theorem \ref{thm:persss} that $X_L$ is algebraically hyperbolic over $L$. Since algebraically hyperbolic projective varieties are $1$-bounded and thus geometrically hyperbolic (Corollary \ref{cor:urata}), the result follows   from Theorem \ref{thm:A}.
\end{proof}

Our second application involves integral points on the moduli space of smooth hypersurfaces. We present the results obtained in \cite{JLitt} in the following section.

 \subsection{The Shafarevich conjecture for smooth hypersurfaces}

We explain in this section how Theorem \ref{thm:A} can be used to show the following finiteness theorem.  This explanation will naturally lead us to studying integral points on moduli spaces. 

\begin{theorem}\label{thm:hypsurf_intro} Let $d\geq 3$     and   $n\geq 2$ be integers. Assume that, for every number field $K$ and every finite set of finite places $S$ of $K$, the set of $\OO_{K,S}$-isomorphism classes of smooth hypersurfaces of degree $d$ in $\mathbb{P}^{n+1}_{\OO_{K,S}}$ is finite. Then, for every $\ZZ$-finitely generated normal integral domain $A$ of characteristic zero, the set of $A$-isomorphism classes of smooth hypersurfaces of degree $d$ in $\mathbb{P}^{n+1}_A$ is finite.
\end{theorem}
 
 To prove Theorem \ref{thm:hypsurf_intro}, we  (i) reformulate its statement   in terms of the arithmetic hyperbolicity of an appropriate moduli space of smooth hypersurfaces, (ii) establish the geometric hyperbolicity of this moduli space, and (iii)   apply Theorem \ref{thm:A}. Indeed, the assumption in Theorem \ref{thm:hypsurf_intro} can be formulated as saying that the (appropriate) moduli space of hypersurfaces is arithmetically hyperbolic over $\Qbar$ and the conclusion of our theorem is then that this moduli space is also arithmetically hyperbolic over larger fields. To make these statements more precise,   let  $\mathrm{Hilb}_{d,n}$ be the Hilbert scheme of smooth hypersurfaces of degree $d$ in $\mathbb{P}^{n+1}$. Note that $\mathrm{Hilb}_{d,n}$ is  a smooth affine scheme over $\mathbb{Z}$.  There is a natural action of the automorphism group scheme $\mathrm{PGL}_{n+2}$ of $\mathbb{P}^{n+1}_{\ZZ}$ on $\mathrm{Hilb}_{d,n}$. Indeed, given a smooth hypersurface $H$ in $\mathbb{P}^{n+1}$ and an automorphism $\sigma$ of $\mathbb{P}^{n+1}$, the resulting hypersurface $\sigma(H)$ is again smooth. 
 
 The quotient of a smooth affine scheme over $\mathbb{Z}$ by a reductive group (such as $\mathrm{PGL}_{n+2}$) is  an affine scheme of finite type over $\ZZ$ by Mumford's GIT.  However, for the study of hyperbolicity and integral points, this quotient scheme is not very helpful, as the action of $\mathrm{PGL}_{n+2}$ on $\mathrm{Hilb}_{d,n}$ is not free. The natural solution it to consider the stacky quotient, as in  \cite{Ben13, BenoistCoarse, JL}. However, one may avoid the use of stacks by adding level structure as in \cite{Javan3}. Indeed,   by \cite{Javan3}, there exists  a smooth affine variety  $H'$ over $\QQ$ with a free action by $\mathrm{PGL}_{n+2, \QQ}$, and a finite \'etale $\mathrm{PGL}_{n+2,\QQ}$-equivariant morphism $H'\to \mathrm{Hilb}_{d,n,\QQ}$.    Let  $U_{d;n}:= \mathrm{PGL}_{n+2,\QQ}\backslash H'$ be the  smooth affine quotient scheme over $\QQ$.  To prove Theorem \ref{thm:hypsurf_intro}, we establish the following result.  
 
 \begin{theorem}\label{thm:hyp_surf_intro2}  Let $d\geq 3$ and $n\geq 2$ be integers.  Assume that   $U_{d;n,\Qbar}$ is arithmetically hyperbolic over $\Qbar$. Then, for every   algebraically closed field $k$ of characteristic zero, the affine variety $U_{d;n,k}$ is arithmetically hyperbolic over $k$.
 \end{theorem}
 \begin{proof} Let us write $U:=U_{d;n,\Qbar}$. The proof relies on a bit of Hodge theory. Indeed, 
we use Deligne's finiteness theorem for monodromy representations \cite{Delignemonodromy}, the infinitesimal Torelli property for smooth hypersurfaces \cite{Fl86}, and the Theorem of the Fixed Part in Hodge theory \cite{Schmid} to show that $U_{k}$ is geometrically hyperbolic over $k$. Then, as  $U_k$ is geometrically hyperbolic over $k$, the result follows from Theorem \ref{thm:A}.  
 \end{proof}
 We now  explain how to deduce Theorem \ref{thm:hypsurf_intro} from Theorem \ref{thm:hyp_surf_intro2}.  

\begin{proof}[Proof of Theorem \ref{thm:hypsurf_intro}]  
Write   $U:=U_{d;n,\Qbar}$.  First, the assumption in Theorem \ref{thm:hypsurf_intro}  can be used to show that $U $ is arithmetically hyperbolic over $\Qbar$.
Then,  since $U$ is arithmetically hyperbolic over $\Qbar$,  it follows from Theorem \ref{thm:hyp_surf_intro2} that $U_k$ is arithmetically hyperbolic for every algebraically closed field   $k$ of characteristic zero. Finally, to conclude the proof,  let us recall that arithmetic hyperbolicity descends along finite \'etale morphisms of varieties (Remark \ref{remark:cw_mordell}). In \cite{JLalg}, the analogous descent statement is proven for finite \'etale morphisms of algebraic stacks, after extending the notion of arithmetic hyperbolicity from schemes to stacks. Thus,   by applying this ``stacky'' Chevalley-Weil theorem to the finite \'etale morphism $U_{d;n,k}\to [\mathrm{PGL}_{n+2, k}\backslash \mathrm{Hilb}_{d,n,k}]$ of stacks, where $[\mathrm{PGL}_{n+2, k}\backslash \mathrm{Hilb}_{d,n,k}]$ denotes the quotient stack, we obtain that  the stack  $[\mathrm{PGL}_{n+2, k}\backslash \mathrm{Hilb}_{d,n,k}]$  is arithmetically hyperbolic over $k$. Finally, the moduli-interpretation of the points of this quotient stack can   be used to see that, for every $\ZZ$-finitely generated normal integral domain $A$ of characteristic zero, the set of $A$-isomorphism classes of smooth hypersurfaces of degree $d$ in $\mathbb{P}^{n+1}_{\OO_{K,S}}$ is finite. This   concludes the proof.
 \end{proof}
 
  \begin{remark}[Period domains]
  Theorem \ref{thm:hyp_surf_intro2} actually follows from a more general statement about varieties with a quasi-finite period map (e.g., Shimura varieties). Namely, in \cite{JLitt} it is shown that a complex algebraic variety with a quasi-finite period map is geometrically hyperbolic.  For other results about period domains we refer the reader to  the article of Bakker-Tsimerman in this book \cite{BakkerTsimermanBook}.
  \end{remark}

\section{Lang's question on openness of hyperbolicity}\label{section:lang_question} 
It is   obvious that being hyperbolic is not stable under specialization. In fact, being pseudo-groupless is not stable under specialization, as a smooth proper curve of genus two can specialize to a tree of $\mathbb{P}^1$'s. Nonetheless, it seems reasonable to suspect that being hyperbolic (resp. pseudo-hyperbolic) is in fact stable under generization. The aim of this section is to investigate this property for all notions of hyperbolicity discussed in these notes.  In fact, on \cite[p.~176]{Lang2} Lang says ``I do not clearly understand the extent to which hyperbolicity is open for the Zariski topology''. This brings us to the following question of Lang and our  starting point of this section.

\begin{question}[Lang]\label{q1}
Let $S$ be a noetherian scheme over $\QQ$ and let $X\to S$ be a projective morphism. Is the set of $s$ in $S$ such that $X_{\overline{k(s)}}$ is groupless a Zariski open subscheme of $S$?
\end{question}

Here we let $k(s)$ denote the residue field of the point $s$, and we let $k(s)\to \overline{k(s)}$ be an algebraic closure of $k(s)$. 
Note that 
one can ask similar questions for the set of $s$ in $S$ such that $X_{\overline{k(s)}}$ is algebraically hyperbolic or arithmetically hyperbolic, respectively.

Before we discuss what one may expect regarding Lang's question,   let us  recall what it means for a subset of a scheme to be a Zariski-countable open.
 
If  $(X, \mathcal{T})$ is a noetherian topological space, then there exists another topology $\mathcal{T}^{\mathrm{cnt}}$, or $\mathcal{T}$-countable, on $X$ whose closed sets are the countable union of $\mathcal{T}$-closed sets. If $S$ is a noetherian scheme, a subset $Z\subset S$ is a Zariski-countable closed if it is a countable union of closed subschemes $Z_1, Z_2, \ldots \subset S$. 

\begin{remark}[What to expect? I] 
We will explain below that the locus of $s$ in $S$ such that $X_s$ is groupless is  a Zariski-countable open of $S$, i.e., its complement is a countable union of closed subschemes. In fact, we will show similar statements for algebraic hyperbolicity, boundedness,  geometric hyperbolicity, and the property of having only subvarieties of general type. Although this provides some indication that the answer to Lang's question might be positive, it is not so clear whether one should expect a positive answer to Lang's question. However,   it seems plausible that, assuming the Strongest Lang-Vojta conjecture (Conjecture \ref{conj:S6}), one can use   certain Correlation Theorems (see Ascher-Turchet \cite{AscherTurchetBook})  to show that  the answer to Lang's question is positive.
\end{remark}

One can also ask about the pseudofied version of Lang's question.

\begin{question}[Pseudo-Lang]\label{q2}
Let $S$ be a noetherian scheme over $\QQ$ and let $X\to S$ be a projective morphism. Is the set of $s$ in $S$ such that $X_{\overline{k(s)}}$ is pseudo-groupless a Zariski open subscheme of $S$?
\end{question}

Again, one can ask similar questions for the set of $s$ in $S$ such that $X_{\overline{k(s)}}$ is pseudo-algebraically hyperbolic or pseudo-arithmetically hyperbolic, respectively.

\begin{remark}[What to expect? II]
We will argue below that one should expect (in light of the Strong Lang-Vojta conjecture) that the answer to the Pseudo-Lang question is positive. This is because of a theorem of Siu-Kawamata-Nakayama on invariance of plurigenera.
\end{remark}

What do we know about the above questions (Questions \ref{q1} and \ref{q2})? The strongest results we dispose of are due to  Nakayama; see \cite[Chapter~VI.4]{NakayamaBook}. In fact, the following theorem can be deduced from Nakayama's \cite[Theorem~VI.4.3]{NakayamaBook}. (Nakayama's theorem is a generalization of theorems of earlier theorems of Siu and Kawamata on invariance of plurigenera.)

\begin{theorem}[Siu, Kawamata, Nakayama]\label{thm:siu}
Let $S$ be a noetherian scheme over $\QQ$ and let $X\to S$ be a projective morphism of schemes. Then, the set of $s$ in $S$ such that $X_s$ is of general type is open in $S$.
\end{theorem}

Thus, by Theorem \ref{thm:siu}, assuming the Strong Lang-Vojta conjecture (Conjecture \ref{conj:S6}), the answer to the Pseudo-Lang question should be positive. Also, assuming the Strong Lang-Vojta conjecture, the set of $s$ in $S$ such that $X_{\overline{k(s)}}$ is pseudo-algebraically hyperbolic should   be open.  Similar statements should hold for pseudo-Mordellicity and pseudo-boundedness. Although neither of these statements are known,  some partial results are obtained in \cite[\S9]{vBJK}.

 In fact, as a consequence of  Nakayama's theorem and  the fact that the stack of proper schemes of general type is a countable union of finitely presented algebraic stacks, one can   prove the  following  result. 

\begin{theorem}[Countable-openness of every subvariety being of general type]\label{thm:specialization}   Let $S$ be a noetherian scheme over $\QQ$ and let $X \to S$ be a projective morphism. Then, the set of $s$ in $S$ such that  every integral closed subvariety of $X_s$  is of general type is Zariski-countable open in $S$.
\end{theorem}

 The countable-openness of the locus of every subvariety being of general type  does not give a satisfying answer to Lang's question. However, it does suggest that every notion appearing in the Lang-Vojta conjecture should be Zariski-countable open. This expectation can be shown to hold for some notions of hyperbolicity. 
  For example,   given a projective morphism of schemes $X\to S$ with $S$ a complex algebraic variety, 
one can show that the locus of $s$ in $S$ such that $X_s$ is algebraically hyperbolic is an open subset of $S$ in the countable-Zariski topology; see \cite{vBJK, Demailly}. This result is essentially due to Demailly.

 \begin{theorem} \label{thm:demailly}
  Let $S$ be a noetherian scheme over $\QQ$ and let $X \to S$ be a projective   morphism.  Then, the set of $s$ in $S$ such that $X_s$ is algebraically hyperbolic  is  Zariski-countable open in $S$.  
\end{theorem}
It is worth noting that this is not the exact result proven by Demailly, as it brings us to a subtle difference between the Zariski-countable topology on a variety $X$ over $\mathbb{C}$ and the induced topology on $X(\CC)$. Indeed, Demailly proved that, if $k=\CC$ and $S^{\textrm{not-ah}}$ is the set of $s$ in $S$ such that $X_s$ is not algebraically hyperbolic, then $S^{\textrm{not-ah}}\cap S(\CC)$ is closed in the countable topology on $S(\CC)$. This, strictly speaking, does not imply that  $S^{\textrm{not-ah}}$ is closed in the countable topology on   $S$. For example, if $S$ is an integral curve over $\CC$ and $\eta$ is the generic point of $S$, then $\{\eta\}$ is not a Zariski-countable open of $S$, whereas $\{\eta\}\cap S(\CC)= \emptyset$ is a Zariski-countable open of $S(\CC)$.

In \cite{vBJK} similar results are obtained for boundedness and geometric hyperbolicity.  The precise statements read as follows.
 \begin{theorem}[Countable-openness of boundedness] \label{thm:demailly-b}
  Let $S$ be a noetherian scheme over $\QQ$  and let $X \to S$ be a projective   morphism.  Then, the set of $s$ in $S$ such that $X_s$ is  bounded  is  Zariski-countable open in $S$.  
\end{theorem}

 \begin{theorem}[Countable-openness of geometric hyperbolicity] \label{thm:demailly-mn}  
  Let $S$ be a noetherian scheme over $\QQ$  and let $X \to S$ be a projective   morphism.  Then, the set of $s$ in $S$ such that $X_s$ is  geometrically hyperbolic  is  Zariski-countable open in $S$.  
\end{theorem}

\begin{remark}[What goes into the proofs of Theorems \ref{thm:demailly}, \ref{thm:demailly-b}, and \ref{thm:demailly-mn}?]\label{remark:ideas}
The main idea behind all these proofs is quite simple. Let us consider Theorem \ref{thm:demailly}.  First, one shows that the  set of $s$ in $S$ such that $X_s$ is \emph{not} algebraically hyperbolic is the image of countably many constructible subsets of $S$. This is essentially a consequence of the fact that  the Hom-scheme between two  projective schemes is a countable union of quasi-projective schemes. Then,   it suffices to note that the set of $s$ in $S$ with $X_s$ algebraically hyperbolic is stable under generization.   This relies on compactness  properties of the moduli stack of stable curves.
\end{remark}

Concerning Lang's question on the locus of groupless varieties, we note that in \cite{JVez} it is shown that the set of $s$ in $S$ such that $X_s$ is groupless is open in the Zariski-countable topology on $S$.

 \begin{theorem}[Countable-openness of grouplessness]\label{thm:group}   Let $S$ be a noetherian scheme over $\QQ$ and let $X \to S$ be a projective morphism. Then, the set of $s$ in $S$ such that   $X_{\overline{k(s)}}$  is  groupless is  Zariski-countable open in $S$.
\end{theorem}

 We finish these notes with a discussion of the proof of Theorem \ref{thm:group}. It will naturally lead us to introducing a non-archimedean counterpart to Lang-Vojta's conjecture.   
 
 \subsection{Non-archimedean hyperbolicity and  Theorem \ref{thm:group}}  
  Let $S$ be a noetherian scheme over $\QQ$ and let $X \to S$ be a projective morphism. Define $S^{n-gr}$ to be the set of $s$ in $S$ such that $X_{\overline{k(s)}}$ is not  groupless.   Our goal is to prove Theorem \ref{thm:group}, i.e., to show that $S^{n-gr}$ is Zariski-countable closed, following the arguments of \cite{JVez}. As is explained in Remark \ref{remark:ideas}, we prove  this   in two steps.
 
 First, one shows that $S^{n-gr}$ is a countable union of constructible subsets.  This step relies on some standard moduli-theoretic techniques. Basically, to say that $X$ is not groupless over $k$ is equivalent to saying that, there is an integer $g$ such that  the Hom-stack $\underline{\Hom}_{\mathcal{A}_g}(\mathcal{U}_g,X\times \mathcal{A}_g)\to \mathcal{A}_g$ has a non-empty fibre over some $k$-point of $\mathcal{A}_g$, where $\mathcal{A}_g$ is the stack of principally polarized $g$-dimensional abelian schemes over $k$, and $\mathcal{U}_g\to \mathcal{A}_g$ is the universal family. We will not discuss this argument and refer the reader to \cite{JVez} for details on this part of the proof.
 
Once the first step is completed, to conclude the proof, it suffices to show that the notion of being groupless is stable under generization. To explain how to do this, we introduce   a new notion of      hyperbolicity for  rigid analytic varieties (and also adic spaces) over a non-archimedean field $K$ of characteristic zero; see \cite[\S2]{JVez}. This notion is inspired by the earlier work of Cherry \cite{Cherry} (see also \cite{AnCherryWang, CherryKoba, CherryRu, LevinWang, LinWang}).
 
 If $K$ is a complete algebraically closed non-archimedean valued field of characteristic zero and $X$ is a finite type scheme over $K$, we let $X^{an}$  be the associated rigid analytic variety over $K$. We say that a variety over $K$ is \emph{$K$-analytically Brody hyperbolic} if, for every finite type connected group scheme $G$ over $K$, every morphism $G^{an}\to X^{an}$ is constant. It follows from this definition  that a $K$-analytically Brody hyperbolic variety is groupless. It seems reasonable to suspect that the converse of this statement holds for projective varieties.

 \begin{conjecture}[Non-archimedean~Lang-Vojta]\label{conj:lang_na}  Let $K$ be an algebraically closed complete non-archimedean valued field of characteristic zero, and let $X$ be a projective variety over $K$.  If $X$ is groupless over $K$, then $X$ is $K$-analytically Brody hyperbolic. 
\end{conjecture}

In  \cite{Cherry} Cherry proves this     conjecture   for  closed subvarieties of abelian varieties. That is, Cherry proved the non-archimedean analogue of the Bloch--Ochiai--Kawamata theorem (Theorem \ref{thm:bok}) for closed subvarieties of abelian varieties.

In \cite{JVez} it is shown that  the above conjecture holds for projective varieties over a non-archimedean field $K$, assuming that  $K$ is of equicharacteristic zero and $X$ is a ``constant'' variety over $K$ (i.e., can be defined over the residue field of $K$).  This actually follows from the following more general result.
 
 \begin{theorem}\label{thm:na}
 Let $K$ be an algebraically closed complete non-archimedean valued field of equicharacteristic zero with valuation ring $\OO_K$, and let $\mathcal{X}\to \Spec \OO_K$ be a proper flat morphism of schemes.   If the special fibre $\mathcal{X}_0$ of $\mathcal{X}\to \Spec \OO_K$ is groupless, then the generic fibre $\mathcal{X}_K$ is $K$-analytically Brody hyperbolic.
 \end{theorem}
 \begin{proof}  
 This is the main result of \cite{JVez} and  is proven in three steps.  Write $X:=\mathcal{X}_K$.
 
 First, one   shows that every morphism $\mathbb{G}_{m,K}^{\an}\to X^{\an}$ is constant  by considering the ``reduction'' map $X^{\an}\to X_0$ and a careful analysis of the residue fields of points  in the image of composed map $\mathbb{G}_{m,K}^{\an}\to X^{\an}\to \mathcal{X}_0$; see \cite[\S 5]{JVez} for details.  This implies that $X $ has  no rational curves.
 
 Now, one wants to show that every morphism $A^{\an}\to X^{\an}$ with $A$ some abelian variety over $K$ is constant.  Instead of appealing to GAGA and trying to use algebraic arguments, we   appeal to the uniformization theorem of Bosch-L\"utkebohmert for abelian varieties. This allows us to reduce to the case that $A$ has good reduction over $\OO_K$. In this reduction step we use that every morphism $\mathbb{G}_{m,K}^{\an}\to X^{\an}$ is constant (which is what we established in the first part of this proof); we refer the reader to \cite[Theorem~2.18]{JVez} for details. 
 
Thus, we have reduced to showing that, for 
$A$ an abelian variety over $K$ \emph{with good reduction over $\OO_K$}, every morphism   $A^{\an}\to X^{\an}$    is constant.   To do so, as $A$ has good reduction over $\OO_K$, we  may   let $\mathcal{A}$ be a smooth proper model for $A$ over $\OO_K$.  Note that the non-constant morphism $A^{\an}\to X^{\an}$ over $K$ algebraizes by GAGA, i.e., it is the analytification of a non-constant morphism $A\to X$. By the valuative criterion of properness, there is a dense open $U\subset \mathcal{A}$ whose complement is of codimension at least two and a morphism $U\to \mathcal{X}$ extending the morphism $A\to X$ on the generic fibre. Now, since $\mathcal{X}_0$ is groupless, it has no rational curves. In particular, as $\mathcal{A}\to \Spec \OO_K$ is smooth,  the morphism $U\to \mathcal{X}$ extends to a morphism $\mathcal{A}\to \mathcal{X}$  by \cite[Proposition~6.2]{GLL}. However, since $\mathcal{X}_0$ is groupless, this morphism is constant on the special fibre. The latter  implies (as $\mathcal{A}\to \Spec \OO_K$ is proper) that the morphism on the generic fibre is constant; see \cite[\S3.2]{JVez} for details. We have shown that, for every abelian variety $A$ over $K$, every morphism $A\to X$ is constant and that every morphism $\mathbb{G}_{m}^{\an}\to X^{\an}$ is constant.
 
 Finally, by adapting the proof of Lemma \ref{lem:why_groupless} one can show that the above  implies that, for every finite type connected group scheme $G$ over $K$, every morphism $G^{\an}\to X^{\an}$ is constant, so that $X$ is $K$-analytically Brody hyperbolic (see \cite[Lemma~2.14]{JVez} for details) .
 \end{proof}
 
To conclude the proof of Theorem \ref{thm:group}, we point out that   a  straightforward  application of Theorem \ref{thm:na}  shows that   being groupless is stable under generization, as required. \qed \\
 
An important problem in the study of non-archimedean hyperbolicity at this moment    is finding   a ``correct'' analogue of the Kobayashi pseudometric (if there is any at all).  Cherry defined an analogue of the Kobayashi metric but it does not have the right properties, as he   showed in \cite{CherryKoba} (see also \cite[\S3.5]{JVez}). A ``correct'' analogue of the Kobayashi metric in the non-archimedean context would most likely have formidable consequences. Indeed, it seems reasonable to suspect that a $K$-analytic Brody hyperbolic projective variety is in fact ``Kobayashi hyperbolic'' over $K$ and that   ``Kobayashi hyperbolic'' projective varieties over $K$ are bounded over $K$ by some version of the Arzel\`a-Ascoli theorem.

 \bibliography{refs_course}{}
\bibliographystyle{plain}

\end{document}